\documentclass{article}

\usepackage{geometry}
\usepackage{amsmath}
\usepackage{amsthm}
\usepackage{amssymb}
\usepackage{mathrsfs}
\usepackage{stmaryrd}
\usepackage[backend=biber,autocite=footnote,style=authoryear]{biblatex}
\usepackage{varioref}
\usepackage{hyperref}
\usepackage[capitalise]{cleveref}

\theoremstyle{plain}
\newtheorem{theorem}{Theorem}
\newtheorem{lemma}{Lemma}

\newtheorem{corollary}{Corollary}
\newtheorem{proposition}{Proposition}

\theoremstyle{remark}
\newtheorem{remark}{Remark}
\newtheorem{definition}{Definition}

\newtheorem{example}{Example}
\newtheorem{question}{Question}

\newcommand{\tripnorm}{\vert
\vert
\vert}
\newcommand{\Bigtripnorm}{\Big\vert
\Big\vert
\Big\vert}
\newcommand{\field}{\mathbb K}
\newcommand{\nat}{\mathbb N}
\newcommand{\real}{\mathbb R}
\newcommand{\complex}{\mathbb C}
\newcommand{\asplundop}{\ensuremath{\mathscr{D}}}
\newcommand{\compactop}{\ensuremath{\mathscr{K}}}
\newcommand{\ssop}{\ensuremath{\mathscr{S}}}
\newcommand{\sepop}{\ensuremath{\mathscr{X}}}
\newcommand{\szlenkop}{\ensuremath{\mathscr{S\negthinspace Z}}}
\newcommand{\allop}{\ensuremath{\mathscr{L}}}
\newcommand{\wcompactop}{\ensuremath{\mathscr{W}}}
\newcommand{\opideal}{\ensuremath{\mathscr{I}}}
\newcommand{\ord}{{\textbf{Ord}}}
\newcommand{\spn}{\mathop{span}}
\newcommand{\diam}{\mathop{diam}}
\newcommand{\dist}{\mathop{dist}}
\newcommand{\MIN}{\mathop{MIN}}
\newcommand{\MAX}{\mathop{MAX}}
\newcommand{\bee}{\mathcal{B}}
\newcommand{\cee}{\mathcal{C}}
\newcommand{\eff}{\mathcal{F}}
\newcommand{\queue}{\mathcal{Q}}
\newcommand{\arr}{\mathcal{R}}
\newcommand{\ess}{\mathcal{S}}
\newcommand{\tee}{\mathcal{T}}
\newcommand{\yoo}{\mathcal{U}}
\newcommand{\dubyoo}{\mathcal{W}}
\newcommand{\scrjay}{\mathscr{J}}
\providecommand{\keywords}[1]{\textbf{\textit{Keywords---}} #1}
\providecommand{\msc}[1]{\textbf{\textit{Mathematics Subject Classification---}} #1}

\begin{document}

\title{Absolutely convex sets of large Szlenk index}

\author{Philip A.H. Brooker}

\maketitle

\begin{abstract}
Let $X$ be a Banach space and $K$ an absolutely convex, weak$^\ast$-compact subset of $X^\ast$. We study consequences of $K$ having a large or undefined Szlenk index and subsequently derive a number of related results concerning basic sequences and universal operators. We show that if $X$ has a countable Szlenk index then $X$ admits a subspace $Y$ such that $Y$ has a basis and the  Szlenk indices of $Y$ are comparable to the Szlenk indices of $X$. If $X$ is separable, then $X$ also admits subspace $Z$ such that the quotient $X/Z$ has a basis and the Szlenk indices of $X/Z$ are comparable to the Szlenk indices of $X$. We also show that for a given ordinal $\xi$ the class of operators whose Szlenk index is not an ordinal less than or equal to $\xi$ admits a universal element if and only if $\xi<\omega_1$; W.B. Johnson's theorem that the formal identity map from $\ell_1$ to $\ell_\infty$ is a universal non-compact operator is then obtained as a corollary. Stronger results are obtained for operators having separable codomain.
\end{abstract}


\msc{46B20, 46B03, 46B15, 47B10}

\section{Introduction}\label{squintro}
The Szlenk index is an ordinal index that measures the difference between the norm and weak$^\ast$ topologies on subsets of dual Banach spaces. It was introduced by Szlenk\autocite{Szlenk1968} to solve (in the negative) the problem of whether there exists a separable, reflexive Banach space whose subspaces exhaust the class of separable, reflexive Banach spaces up to isomorphism. Since then the Szlenk index has found many uses in the study of Banach spaces and their operators, as outlined in the surveys of \textcite{Lancien2006} and \textcite{Rosenthal2003}.

In the current paper we study the Szlenk index in two main contexts, the first of these being the theory of basic sequences in Banach spaces. Our work on basic sequences and the Szlenk index is based on the classical method of Mazur\autocite[See e.g.][Section~1.5]{Albiac2016} for producing subspaces with a basis and on the more recent method of Johnson and Rosenthal\autocite{Johnson1972} for producing quotients with a basis (a dual version of the Mazur technique). We extend previous work in this area by Lancien\autocite{Lancien1996} and Dilworth-Kutzarova-Lancien-Randrianarivony\autocite{Dilworth2017}. The other main motivation for our paper is the study of \emph{universal} operators. In \cref{structureback,uniback} we provide relevant background on these two areas of study and explain the main contributions of the current paper to these topics, along the way providing an outline of the content of \cref{mainlemmasection,basissection,universalsection,kneesbees,nonspear}. Basic notation and terminology for Banach space theory is set out in \cref{notation}, whilst background information on more specialised topics such as Szlenk indices, trees and operators acting on Banach spaces over trees will be provided in \cref{backgrounder}.

\subsection{Structure of Banach spaces with large Szlenk index}\label{structureback}
Several authors have already studied the structure of subspaces and quotients of Banach spaces having Szlenk index larger then a given ordinal. We shall now sketch the previous such results of most relevance to us and explain the contributions of the current paper to this topic. The first result of interest to us is the following result due to G.~Lancien\autocite[Proposition~3.1]{Lancien1996}.

\begin{proposition}\label{margeincharge}
Let $X$ be a Banach space and $\xi<\omega_1$. If $Sz(X)>\xi$ then there exists a separable subspace $Y$ of $X$ such that $Sz(Y)>\xi$.
\end{proposition}

To prove \cref{margeincharge} Lancien showed\autocite[Lemma~3.4]{Lancien1996} that for a suitable tree $\tee$ of rank $\xi+1$, the estimate $Sz(X,\epsilon)>\xi$ (which follows for some $\epsilon>0$ from the estimate $Sz(X)>\xi$) implies the existence of families of vectors $(x_t)_{t\in\tee}\subseteq B_X$ and $(x_t^\ast)_{t\in\tee}\subseteq B_{X^\ast}$ whose construction depends on $\epsilon$ and which satisfy certain properties that bear witness to the fact that $Sz(X)>\xi$. Without giving the precise details of Lancien's construction here, we mention that the subspace $Y$ is taken to be the closed linear span in $X$ of the family $(x_t)_{t\in\tee}$. Recently, Dilworth, Kutzarova, Lancien and Randrianarivony\autocite[Proposition~3.1(i)]{Dilworth2017} have adapted Lancien's construction to show that, under the additional hypothesis that $X$ is reflexive, $(x_t)_{t\in\tee}$ may be assumed to be a basic sequence for a suitable enumeration of $\tee$. In \cref{mainlemmasection} we devise a further refinement of Lancien's construction, the precise details of which are captured in the statement of our main technical result, \cref{oopsblah}. A number of consequences of \cref{oopsblah} follow, including the main results of the current paper, namely \cref{firstbasistheorem,drainyourflagon,countuniv,sepseekay}.

Our first application of \cref{oopsblah} is the following result, proved in \cref{basissection}, which asserts that if $X$ is an infinite-dimensional Banach space with countable Szlenk index, then $X$ admits a subspace with a basis and with $\epsilon$-Szlenk indices comparable to the $\epsilon$-Szlenk indices of $X$: 

\begin{theorem}\label{firstbasistheorem}
Let $X$ be an infinite-dimensional Banach space such that $Sz(X)<\omega_1$ and let $\delta>0$. Then there exists a subspace $Y\subseteq X$ such that $Y$ has a shrinking basis with basis constant not exceeding $1+\delta$ and such that
\begin{equation}\label{wuppertare}
\forall\,\epsilon>0\quad Sz\Big(Y,\frac{\epsilon}{66}\Big)\geq Sz(X,\epsilon),
\end{equation}
hence $Sz(Y)=Sz(X)$.
\end{theorem}

To compare \cref{firstbasistheorem} with the earlier results of \textcite{Lancien1996} and \textcite{Dilworth2017}, first note that it follows easily from \cref{margeincharge} that if $X$ is a Banach space with $Sz(X)<\omega_1$, then there exists a separable subspace $Y$ of $X$ such that $Sz(Y)=Sz(X)$. For instance, take $Y$ to be the closed linear span of $\bigcup_{n=1}^\infty Y_n$ in $X$, where, for each $n\in\nat$, $Y_n$ is a separable subspace of $X$ with $Sz(Y_n)>Sz(X,1/n)-1$. However, even in the case that $X$ is reflexive one cannot in general conclude from Proposition~3.1(i) of \textcite{Dilworth2017} (or from \cref{margeincharge}) that $X$ admits a subspace with a basis \emph{and} with Szlenk index equal to the Szlenk index of $X$. Indeed, from Proposition~3.1(i) of \textcite{Dilworth2017} we may only deduce that the subspace $Y$ with a basis satisfies $Sz(Y)\geq \xi\omega$ (on the other hand, if $X$ is reflexive and $Sz(X)=\omega^{\alpha+1}$ for some $\alpha<\omega_1$, then it follows from Proposition~3.1(i) of \textcite{Dilworth2017} that $X$ admits a subspace $Y$ with a basis and satisfying $Sz(Y)=Sz(X)$; to see this, take a subspace $Y$ with a basis such that $Sz(Y)>\omega^\alpha$).

One of the most important features of \cref{oopsblah} is that the construction of families of vectors $(x_t)$ and $(x_t^\ast)$ in the proof of the theorem depends on a family $(\epsilon_n)_{n<\omega}$ of (possibly infinitely many distinct) epsilons and associated estimates $Sz(K,\epsilon_n)>\xi_n$; the earlier such constructions from \textcite{Lancien1996} and \textcite{Dilworth2017} depend instead on consideration of just a single $\epsilon$ and an associated estimate $Sz(X,\epsilon)>\xi$. The greater generality taken in our approach plays a key role in the proof of \cref{firstbasistheorem}. In particular whereas the constructions underpinning \cref{margeincharge} above and Proposition~3.1(i) of \textcite{Dilworth2017} yield an estimate of the type shown at \cref{wuppertare} for single $\epsilon>0$, with our approach we take the sequence $(\epsilon_n)_{n<\omega}$ in the statement of \cref{oopsblah} to be dense in $(0,\infty)$ and apply a straightforward density argument to obtain the estimate shown at \cref{wuppertare} for all $\epsilon>0$, which is enough to conclude the final assertion of \cref{firstbasistheorem} (without any reflexivity hypothesis).

We now turn our attention to quotients. Under the additional assumption that $X$ is separable, \cref{oopsblah} yields the existence of a quotient of $X$ having a basis and with $\epsilon$-Szlenk indices comparable to the $\epsilon$-Szlenk indices of $X$, as per the following result established in \cref{basissection}. 

\begin{theorem}\label{drainyourflagon}
Let $X$ be an infinite-dimensional Banach space with separable dual and let $\delta>0$. Then there exists a subspace $Z\subseteq X$ such that $X/Z$ has a shrinking basis with basis constant not exceeding $1+\delta$ and such that
\begin{equation}\label{scupperscare}
\forall\,\epsilon>0\quad Sz\Big(X/Z,\frac{\epsilon}{9}\Big)\geq Sz(X,\epsilon)\,,
\end{equation}
hence $Sz(X/Z)=Sz(X)$.
\end{theorem}

\cref{drainyourflagon} should be compared with the following earlier result of G.~Lancien\autocite[Proposition~3.5]{Lancien1996}.
\begin{proposition}\label{bargeincharge}
Let $X$ be a separable Banach space and $\xi<\omega_1$. If $Sz(X)>\xi$ then there exists a subspace $Z$ of $X$ such that $X/Z$ has a shrinking basis and $Sz(X/Z)>\xi$.
\end{proposition}

To prove \cref{bargeincharge} Lancien combined the construction developed in the proof of Lemma~3.4 of \textcite{Lancien1996} (previously mentioned in the discussion following \vref{margeincharge}) with the techniques developed by Johnson and Rosenthal\autocite{Johnson1972} for constructing weak$^\ast$-basic sequences in dual Banach spaces. We take a similar approach in the proof of \cref{oopsblah}, with the main difference between \cref{drainyourflagon} and \cref{bargeincharge} again arising from the more general approach taken in the construction of the families of vectors $(x_t)$ and $(x_t^\ast)$ in the proof of the \cref{oopsblah}, in particular the dependence on an infinite family $(\epsilon_n)_{n<\omega}$ of (possibly distinct) epsilons and associated estimates $Sz(K,\epsilon_n)>\xi_n$ rather than just a single $\epsilon$ and associated estimate $Sz(X,\epsilon)$.

Our comparison of \cref{drainyourflagon} with \cref{bargeincharge} is similar to our comparison of \cref{firstbasistheorem} with \cref{margeincharge}. In general one may deduce from \cref{bargeincharge} only that $Sz(X/Z)\geq \xi\omega$; the stronger assertion that $Sz(X/Z)=Sz(X)$ need not follow except in the case that $Sz(X)=\omega^{\alpha+1}$ for some ordinal $\alpha$ and we take $Z$ to satisfy $Sz(X/Z)>\omega^\alpha$. By contrast, \cref{drainyourflagon} provides the conclusion that $Sz(X/Z)=Sz(X)$ in all cases.

\begin{remark}
Trees have been used by a number of authors to study and compare a variety of ordinal indices on Banach spaces. For instance, the work of \textcite{Alspach2005} defines and studies a number of local $\ell_1$-indices and compares these indices with one another and with the Szlenk index.
\end{remark}

\subsection{Universal operators}\label{uniback}
The initial motivation for writing the current paper was a desire to study the Szlenk index in the context of the problem of finding universal elements for certain subclasses of the class $\allop$ of all (bounded, linear) operators between Banach spaces. For operators $T\in\allop(X,Y)$ and $S\in\allop(W,Z)$, where $W,X,Y$ and $Z$ are Banach spaces, we say that $S$ \emph{factors through} $T$ (or, equivalently, that $T$ \emph{factors} $S$) if there exist $U\in \allop(W,X)$ and $V\in\allop(Y,Z)$ such that $VTU=S$. With this terminology, for a given subclass $\mathscr{C}$ of $\allop$ we say that an operator $\Upsilon\in\mathscr{C}$ is \emph{universal for} $\mathscr{C}$ if $\Upsilon$ factors through every element of $\mathscr{C}$. Typically $\mathscr{C}$ will be the complement $\complement\opideal$ of an operator ideal $\opideal$ in the sense of Pietsch\autocite{Pietsch1980} (that is, $\complement\opideal$ consists of all elements of $\allop$ that do not belong to $\opideal$), or perhaps the restriction $\mathscr{J}\cap \complement\opideal$ of $\complement\opideal$ to a large subclass $\mathscr{J}$ of $\allop$; e.g., $\mathscr{J}$ might denote a large operator ideal or the class of all operators having a specified domain or codomain. One may think of a universal element of the class $\mathscr{C}$ as a minimal element of $\mathscr{C}$ that is `fixed' or `preserved' by each element of $\mathscr{C}$. 

The notion of universality for a class of operators goes back to the work of Lindenstrauss and {Pe{\l}czy{\'n}ski}\autocite[Theorem~8.1]{Lindenstrauss1968}, who obtained the following result.

\begin{theorem}\label{LPthm}
Let $X$ and $Y$ be Banach spaces and suppose $T:X\longrightarrow Y$ is a non-weakly compact operator. Then $T$ factors the (non-weakly compact) summation operator $\Sigma: (a_n)_{n=1}^\infty \mapsto (\sum_{i=1}^n a_i)_{n=1}^\infty$ from $\ell_1$ to $\ell_\infty$. In particular, $\Sigma$ is universal for the class of non-weakly compact operators. 
\end{theorem}

Since the publication of \textcite{Lindenstrauss1968} a number of results in a similar spirit to \cref{LPthm} have appeared in the literature. Perhaps the most well-known is the following result of W.B.~Johnson\autocite{Johnson1971a}, which is a special case of \cref{countuniv} of the current paper.

\begin{theorem}\label{Jthm}
Let $X$ and $Y$ be Banach spaces and suppose $T:X \longrightarrow Y$ is a non-compact operator. Then $T$ factors the (non-compact) formal identity operator from $\ell_1$ to $\ell_\infty$. In particular, the formal identity operator from $\ell_1$ to $\ell_\infty$ is universal for the class of non-compact operators. 
\end{theorem}

Another universality result of note, due to C. Stegall\autocite[Theorem~4]{Stegall1975}, is the existence of a universal non-Asplund operator. The Asplund operators have several equivalent definitions in the literature; in the current paper we say that an operator $T:X\longrightarrow Y$ is \textit{Asplund} if $(T\vert_Z)^\ast (Y^\ast)$ is separable for any separable subspace $Z\subseteq X$. We refer the reader to \textcite{Stegall1981} for further properties and characterisations of Asplund operators. Stegall's universal operator is defined in terms of the Haar system $(h_m)_{m=0}^\infty\subseteq C(\{ 0,1\}^\omega)$, where each factor $\{ 0,1\}$ is discrete and $\{0,1\}^\omega$ is equipped with its compact Hausdorff product topology. For the purpose of stating Stegall's result, we let $\mu$ denote the product measure on $\{0,1\}^\omega$ obtained by equipping each factor $\{ 0,1\}$ with its discrete uniform probability measure and let $H:\ell_1\longrightarrow L_\infty(\{0,1\}^\omega,\mu)$ be defined by setting $Hx=\sum_{m=1}^\infty x(m)h_{m-1}$ for each $x=(x(m))_{m=1}^\infty\in \ell_1$. Stegall's result is the following theorem.
\begin{theorem}\label{thewox}
Let $X$ and $Y$ be Banach spaces such that $X$ is separable and suppose $T\in\allop(X, Y)$ is such that $T^\ast(Y^\ast)$ is nonseparable. Then $H$ factors through $T$.
\end{theorem}
Since the domain of $H$, namely $\ell_1$, is separable, and since $H^\ast$ has non-separable range, $H$ is a non-Asplund operator. It therefore follows from \cref{thewox} that $H$ is a universal non-Asplund operator. We note that a different universal non-Asplund operator has been obtained by the author of the current paper using the techniques developed here in the context of studying the Szlenk index\autocite{Brooker2017}.

Other universality results besides those mentioned above can be found in the work of \textcite{Brooker2017}, \textcite{Cilia2015}, \textcite{Dilworth1985}, \textcite{Girardi1997}, \textcite{Hinrichs2000}, \textcite{Oikhberg2016}, and the Handbook survey on operator ideals by \textcite{Diestel2001}. \cref{Jthm} above has been applied in the study of information-based complexity by \textcite{Hinrichs2013}.

\begin{remark}
K.~Beanland and R.M.~Causey\autocite[arXiv:1711.09244]{BeanlandCausey} have recently attained universality results for further classes of operators defined in terms of operator ideals.
\end{remark}

In \cref{universalsection} we turn our attention to applying \cref{oopsblah} to the study of universal operators for classes defined in terms of a bound on the Szlenk index. Our main result to this end is the following theorem. 

\begin{theorem}\label{countuniv} 
Let $X$ and $Y$ be Banach spaces, $\alpha<\omega_1$, $T\in\allop(X,Y)\setminus \szlenkop_\alpha(X,Y)$ and $(\tee,\preceq)$ a countably infinite, rooted, well-founded tree with $\rho(\tee)<\omega^{\alpha+1}$. Then $\Sigma_{\tee^\star}$ factors through $T$. Moreover if $\tee$ is blossomed and $\rho(\tee)\geq \omega^\alpha$, then $\Sigma_{\tee^\star}$ is universal for $\complement\szlenkop_\alpha$. It follows that, for an ordinal $\beta$, the class $\complement\szlenkop_\beta$ admits a universal element if and only if $\beta<\omega_1$.
\end{theorem}

In the statement of \cref{countuniv} $\szlenkop_\alpha$ denotes the class consisting of all operators whose Szlenk index is an ordinal no larger than $\omega^\alpha$ (c.f. \vref{szideal}), whilst the operator $\Sigma_{\tee^\star}$ is the path-sum operator of $\tee^\star$ defined in \cref{jeefberky}, where $\tee^\star$ is the rooted tree $\tee$ minus its (unique) minimal element. 

Stronger universality results are achieved in \cref{kneesbees} for operators having separable codomain, the main result in this direction being the following theorem.

\begin{theorem}\label{sepseekay}
Let $X$ and $Y$ be Banach spaces, $\alpha<\omega_1$, $T\in \allop(X,Y)\setminus \szlenkop_\alpha(X,Y)$ and $(\tee,\preceq)$ a countably infinite, rooted, well-founded tree with $\rho(\tee)<\omega^{\alpha+1}$. If $Y$ is separable then $\mathring{\sigma}_{\tee}$ factors through $T$. Moreover if $\tee$ is blossomed and $\rho(\tee)\geq\omega^\alpha$ then $\mathring{\sigma}_{\tee}$ is universal for the class of non-$\alpha$-Szlenk operators having separable codomain.
\end{theorem}

The operator $\mathring{\sigma}_{\tee}$ featured in the statement of \cref{sepseekay} is the path-sum operator defined at \vref{trootedree}, which is essentially the same as $\Sigma_\tee$ but with a restricted (in particular, norm separable) codomain. We conclude \cref{kneesbees} by noting that stronger versions of the classical universality results of Lindenstrauss and {Pe{\l}czy{\'n}ski} (\vref{LPthm}) and Johnson (\vref{Jthm}) hold for operators whose codomain has weak$^\ast$-sequentially compact dual ball.

In Section \cref{nonspear} we investigate whether the techniques developed in \cref{mainlemmasection,universalsection} can be used to show the existence of universal operators for classes of Asplund operators having an uncountable strict lower bound for the Szlenk index. Though we do not completely answer this question, we show that the techniques developed in the earlier sections of the paper cannot decide the existence of universal operators in this setting in ZFC.

\subsection{Notation and terminology}\label{notation}
We now outline the notation and terminology used in the current paper. We work with Banach spaces over the scalar field $\field=\real$ or $\complex$. Typical Banach spaces are denoted by the letters $W$, $X$, $Y$ and $Z$, with the identity operator of $X$ denoted $Id_X$. Following the terminology introduced earlier in the current paper, for Banach spaces $W$, $X$ and $Z$ and an operator $S:W\longrightarrow Z$, we say that $S$ \emph{factors through} $X$ if $S$ factors through $Id_X$. We write $X^\ast$ for the dual space of $X$ and denote by $\imath_X$ the canonical embedding of $X$ into $X^{\ast\ast}$. We define $B_X:=\{ x\in X \mid\Vert x\Vert\leq 1\}$ and $B_X^\circ:= \{ x\in X \mid\Vert x\Vert< 1\}$. By a \textit{subspace} of a Banach space $X$ we mean a linear subspace of $X$ that is closed in the norm topology. For a Banach space $X$, $C\subseteq X$ and $D\subseteq X^\ast$ we define $C^\perp:= \{x^\ast \in X^\ast\mid \forall x\in C, \, x^\ast (x) =0 \}$ and $D_\perp = \{ x\in X\mid \forall x^\ast\in D, \, x^\ast (x) =0 \}$. We denote by $[C]$ the norm closed linear hull of $C$ in $X$, with a typical variation on this notation being that for an indexed set $\{ x_i\mid i\in I \}\subseteq X$ we may write $[x_i]_{i\in I}$ or $[x_i\mid i\in I]$ in place of $[\{x_i\mid i\in I\}]$. With this notation we have that $\overline{[D]}^{w^\ast}$ is the weak$^\ast$ closed linear hull of $D$ in $X^\ast$. We shall make use of the well-known fact that, for a Banach space $X$ and a sequence $(x_m^\ast)_{m=1}^\infty\subseteq X^\ast$, the quotient map $Q: X\longrightarrow X/\bigcap_{m=1}^\infty\ker(x_m^\ast)$ has the property that $Q^\ast$ is an isometric weak$^\ast$-isomorphism of $(X/\bigcap_{m=1}^\infty\ker(x_m^\ast))^\ast$ onto $\overline{[x_m^\ast\mid m\in\nat]}^{w^\ast}$. 

Operator ideals are denoted by script letters such as $\opideal$. Operator ideals of particular interest in the current paper are:
\begin{itemize}
\item $\compactop$, the \textit{compact} operators;
\item $\wcompactop$, the \textit{weakly compact} operators;
\item $\sepop$, the operators having separable range;
\item $\sepop^\ast$, the operators whose adjoint has separable range;
\item $\ssop$, the strictly singular operators;
\item $\asplundop$, the \textit{Asplund} operators (also known as the \textit{decomposing} operators); and,
\item $\szlenkop_\alpha$, the $\alpha$-Szlenk operators for a given ordinal $\alpha$.
\end{itemize}
All of the operator ideals in the list above are closed, and most of them are well known. An operator $T:X\longrightarrow Y$ is \emph{strictly singular} if for every infinite dimensional subspace $Z\subseteq X$ the restriction $T|_Z$ fails to be an isomorphic embedding of $Z$ into $Y$ (that is, if for every such $Z$ and every $\epsilon>0$ there exists $z\in Z$ such that $\Vert z\Vert =1$ and $\Vert Tz\Vert <\epsilon$). For a given ordinal $\alpha$, the class $\szlenkop_\alpha$ consists of all operators whose Szlenk index is an ordinal not exceeding $\omega^\alpha$. These classes have been studied in some detail by the current author\autocite{Brooker2012} and we note in particular that important relationships between the operator ideals $\szlenkop_\alpha$ and other ideals in the list above shall be given in \cref{szlenksection}. It is well known\autocite[See e.g.][Proposition~4.4.8]{Pietsch1980} that $\sepop^\ast$ is a subclass of $\sepop$.

For infinite dimensional Banach spaces $X$ and $Z$ we say that $X$ is \emph{$Z$-saturated} if every infinite dimensional subspace of $X$ contains a further subspace that is isomorphic to $Z$. It is well known\autocite[See e.g.][Chapter~2]{Albiac2016} that for each $1\leq p<\infty$ the Banach space $\ell_p$ is $\ell_p$-saturated and that for distinct $p,q\in [1,\infty)$ neither of the spaces $\ell_p$ and $\ell_q$ is isomorphic to a subspace of the other. From this is follows readily that if $X$ is $\ell_p$-saturated and $Y$ is $\ell_q$-saturated for distinct $p,q\in [1,\infty)$ then every operator from $X$ to $Y$ is strictly singular (we shall use this fact later in the proof of \cref{otherclasses}).

By $\ord$ we denote the class of all ordinals, so that by $\alpha\in\ord$ we mean that $\alpha$ is an ordinal. We write $cof(\alpha)$ for the cofinality of the ordinal $\alpha$. If $\alpha$ is a successor ordinal, we write $\alpha-1$ to mean the unique ordinal whose successor is $\alpha$. 

For a set $S$ and a subset $R\subseteq S$ we write $\chi^S_R$ for the indicator function of $R$ in $S$, or simply $\chi_R$ if no confusion can result. When discussing a Banach space $\ell_1(S)$ for some set $S$, for $s\in S$ we typically denote by $e_s$ the element of $\ell_1(S)$ satisfying $e_s(s')=1$ if $s'=s$ and $e_s(s') = 0$ if $s'\neq s$ ($s'\in S$).We thus denote by $(e_n)_{n=1}^\infty$ the standard unit vector basis of $\ell_1=\ell_1(\nat)$. Where confusion may otherwise result, we may write $e_s^S$ in place of $e_s$ to specify the space $\ell_1(S)$ to which $e_s$ belongs.

We shall repeatedly use the fact that for a set $I$, Banach space $X$ and family $\{ x_i\mid i\in I\}\subseteq X$ with $\sup_{i\in I}\Vert x_i\Vert <\infty$, there exists a unique element of $\allop(\ell_1(I), X)$ satisfying $e_i\mapsto x_i$, $i\in I$.

For a Banach space $X$, a subset $A\subseteq X$, and $\epsilon>0$, we say that $A$ is \textit{$\epsilon$-separated} if $\Vert x-y\Vert>\epsilon$ for any distinct $x,y\in A$. For $B\subseteq C\subseteq X$ and $\delta>0$ we say that $B$ is a \textit{$\delta$-net in $C$} if for every $w\in C$ there exists $z\in B$ such that $\Vert w-z\Vert \leq \delta$.

\section{The Szlenk index, trees and operators on Banach spaces over trees}\label{backgrounder}

\subsection{The Szlenk index}\label{szlenksection}
Let $X$ be a Banach space. For each $\epsilon>0$ define a derivation $s_\epsilon$ on weak$^\ast$-compact subsets of $X^\ast$ as follows: for weak$^\ast$-compact $K\subseteq X^\ast$ let
\[
s_\epsilon(K):=\{ x^\ast \in K\mid \diam(\yoo\cap K)>\epsilon \mbox{ for every weak}^\ast\mbox{-open }\yoo\ni x^\ast\}\,.
\]
Iterate $s_\epsilon$ transfinitely by setting $s_\epsilon^0(K)=K$, $s_\epsilon^{\xi+1}(K) = s_\epsilon(s_\epsilon^\xi(K))$ for every ordinal $\xi$, and $s_\epsilon^\xi(K) =\bigcap_{\zeta<\xi}s_\epsilon^\zeta(K)$ whenever $\xi$ is a limit ordinal. The \textit{$\epsilon$-Szlenk index of $K$}, denoted $Sz(K,\epsilon)$, is defined as the smallest ordinal $\xi$ such that $s_\epsilon^\xi(K)=\emptyset$, if such an ordinal exists; if no such ordinal exists then $Sz(K,\epsilon)$ is undefined. (Note that, by weak$^\ast$-compactness, $Sz(K,\epsilon)$ is a successor ordinal when it exists.) Notationally, we write $Sz(K,\epsilon)<\infty$ to mean that $Sz(K,\epsilon)$ is defined, and $Sz(K,\epsilon)=\infty$ to mean that $Sz(K,\epsilon)$ is undefined. If $Sz(K,\epsilon)$ is defined for all $\epsilon>0$ then the \textit{Szlenk index of $K$}, denoted $Sz(K)$, is the ordinal $\sup_{\epsilon>0}Sz(K,\epsilon)$. If $Sz(K,\epsilon)$ is undefined for some $\epsilon>0$, then $Sz(K)$ is undefined; we write $Sz(K)<\infty$ to mean that $Sz(K)$ is defined, and $Sz(K)=\infty$ to mean that $Sz(K)$ is undefined. Note that while $Sz(K,\epsilon)\leq \xi$ means that $Sz(K,\epsilon)$ is defined and equal to an ordinal not exceeding $\xi$, the statement $Sz(K,\epsilon)\nleq \xi$ means either that $Sz(K,\epsilon)$ is undefined or that $Sz(K,\epsilon)$ is defined and exceeds $\xi$; similarly, $Sz(K)\nleq\xi$ means either that $Sz(K)$ is undefined or that $Sz(K)$ is defined and equal to an ordinal exceeding $\xi$.

Define the \textit{$\epsilon$-Szlenk index of $X$} and the \textit{Szlenk index of $X$} to be the indices $Sz(X,\epsilon):=Sz(B_{X^\ast},\epsilon)$ and $Sz(X):=Sz(B_{X^\ast})$, respectively. If $Y$ is a Banach space and $T:X\longrightarrow Y$ an operator, define the \textit{$\epsilon$-Szlenk index of $T$} and the \textit{Szlenk index of $T$} to be the indices $Sz(T,\epsilon):=Sz(T^\ast (B_{X^\ast}),\epsilon)$ and $Sz(T):=Sz(T^\ast (B_{X^\ast}))$, respectively.

A useful survey of the Szlenk index and its applications in Banach space theory has been written by \textcite{Lancien2006}. For facts regarding Szlenk indices of operators we refer the reader to the work of \textcite{Brooker2012}. The following proposition collects some well-known facts concerning Szlenk indices of Banach spaces and operators.

\begin{proposition}
Let $X$ be a Banach space.
\begin{itemize}
\item[(i)] For $K_1\subseteq K_2\subseteq X^\ast$, $\epsilon_1\geq\epsilon_2>0$ and ordinals $\xi_1\geq\xi_2$ we have $s_{\epsilon_1}^{\xi_1}(K_1)\subseteq s_{\epsilon_2}^{\xi_2}(K_2)$.
\item[(ii)] For a subspace $Z\subseteq X$ we have $Sz(Z)\leq Sz(X)$ and $Sz(X/Z)\leq Sz(X)$.
\item[(iii)] The following are equivalent:
\begin{itemize}
\item[(a)] $Sz(X)<\infty$ (that is, the Szlenk index is defined).
\item[(b)] $X$ is an Asplund space.
\item[(c)] $X^\ast$ has the Radon-Nikod{\'y}m property.
\item[(d)] Every separable subspace of $X$ has separable dual.
\end{itemize}
\end{itemize}
\end{proposition}

An argument due to A. Sersouri\autocite[][proof of Lemma~6]{Sersouri1989} shows that if the Szlenk index of a Banach space or an operator is defined, then it is of the form $\omega^\alpha$ for some ordinal $\alpha$ (and more recently Causey\autocite{Causey2017} has described all the possible ordinals $\alpha$ for which $\omega^\alpha$ may be realised as the Szlenk index of a Banach space); this observation leads to the following definition.

\begin{definition}\label{szideal}
For $\alpha$ an ordinal define the class
\[
\szlenkop_\alpha:=\{ T\in\allop\mid Sz(T)\leq\omega^\alpha\}\,.
\]
If $T\in\szlenkop_\alpha$, we say that $T$ is \textit{$\alpha$-Szlenk}.
\end{definition}

It is a result of \textcite{Brooker2012} that the classes $\szlenkop_\alpha$ are distinct for different values of $\alpha$ and that each such class is a closed operator ideal. Moreover, $\szlenkop_0$ coincides with the class $\compactop$ of compact operators, whilst the class $\bigcup_{\alpha\in\ord}\szlenkop_\alpha$ of all operators whose Szlenk index is defined coincides with the class $\asplundop$ of Asplund operators. For operators with separable range, the following result of Brooker\autocite[][Proposition~2.11]{Brooker2012} provides information regarding the relationship between the classes $\sepop^\ast$, $\asplundop$ and $\szlenkop_\alpha$ for $\alpha\in\ord$.

\begin{proposition}\label{sepsepseppy} The following chain of equalities holds:
\begin{equation*} \displaystyle \sepop^\ast = \sepop \cap \asplundop =  \sepop\cap\bigcup_{\alpha \in\ord} \szlenkop_\alpha = \sepop \cap \bigcup_{\alpha < \omega_1}\szlenkop_\alpha = \sepop \cap \szlenkop_{\omega_1}.\end{equation*}
\end{proposition}

\subsection{Trees}\label{treesubsection}
A \textit{tree} is a partially ordered set $(\tee,\preceq)$ for which the set $\{ s\in\tee\mid s\preceq t\}$ is well-ordered for every $t\in\tee$. We shall frequently suppress the partial order $\preceq$ and refer to the underlying set $\tee$ as the tree. An element of a tree is called a \textit{node}. For $\ess\subseteq\tee$ we denote by $\MIN(\ess)$ (resp., $\MAX(\ess)$) the set of all minimal (resp., maximal) elements of $\ess$. A \textit{subtree} of $\tee$ is a subset of $\tee$ equipped with the partial order induced by the partial order of $\tee$, which we also denote $\preceq$. A \textit{chain} in $\tee$ is a totally ordered subset of $\tee$. A \textit{branch} of $\tee$ is a maximal (with respect to set inclusion) totally ordered subset of $\tee$. We say that $\tee$ is \textit{well-founded} if it contains no infinite branches, and \textit{chain-complete} if every chain $\cee$ in $\tee$ admits a unique least upper bound. Clearly, every well-founded tree is chain-complete. A subset $\ess\subseteq\tee$ is said to be \textit{downwards closed in} $\tee$ if $\ess=\bigcup_{t\in\ess}\{ s\in\tee\mid s\preceq t\}$. Following \textcite{Todorcevic1984}, a \textit{path} in $\tee$ is a downwards closed, totally ordered subset of $\tee$. An \textit{interval} in $\tee$ is a subset of $\tee$ of the form $(t',t'']$, $[t',t'']$, $[t',t'')$ or $(t',t'')$, where, for $t',t''\in\tee$, $(t',t'']:=\{ t\in\tee\mid t'\prec t\preceq t''\}$ and the other types of intervals are defined analogously. (For a tree $(\tee,\preceq)$ and $s,t\in\tee$ we write $s\prec t$ to mean that $s\preceq t$ and $s\neq t$.) For $t\in\tee$ we define the following sets:
\begin{align*}
\tee[\preceq t]&= \{ s\in\tee\mid s\preceq t\}\\
\tee[\prec t]&= \{ s\in\tee\mid s\prec t\}\\
\tee[t \preceq]&= \{ s\in\tee\mid t\preceq s\}\\
\tee[t\prec]&= \{ s\in\tee\mid t\prec s\}\\
\tee[t+]&= \MIN(\tee[t\prec])
\end{align*}
By $t^-$ we denote the maximal element of $\tee[\prec t]$, if it exists (that is, if the order type of $\tee[\prec t]$ is a successor). If $s,t\in\tee$ are such that $s\npreceq t$ and $t\npreceq s$, then we write $s\perp t$. Following \textcite{Godefroy2001}, a subtree $\ess$ of $\tee$ is said to be a \textit{full} subtree of $\tee$ if it is downwards closed, $\vert \ess\cap \MIN(\tee) \vert=\vert \MIN(\tee)\vert$, and for every $t\in\ess$ we have $\vert \ess[t+]\vert = \vert \tee[t+]\vert$. A tree is said to be \textit{rooted} if $\vert \MIN(T)\vert \leq 1$. In particular, a nonempty tree is rooted if and only if it admits a unique minimal element, which we call the \textit{root} of $\tee$. We denote by $\tee^\star$ the subtree $\tee\setminus\MIN(\tee)$ of $\tee$. For $t\in\tee$ the \textit{height of $t$}, denoted $ht_\tee(t)$, is the order type of $\tee[\prec t]$. The \textit{height} of $\tee$ is the ordinal $ht(\tee)=\sup\{ ht_\tee(t)+1\mid t\in\tee\}$. Note that $ht(\tee)\leq\omega$ if and only if $\tee[\prec t]$ is finite for every $t\in\tee$.

Let $\tee= (\tee,\preceq)$ be a tree, $\alpha$ an ordinal and $\psi:\alpha\longrightarrow\tee$ a surjection. Then $\psi$ induces a well-ordering of $\tee$ that extends $\preceq$. Indeed, define $A_0 = \tee[\preceq \psi(0)]$ and, if $\beta>0$ is an ordinal such that $A_\gamma$ has been defined for all $\gamma<\beta$, define $A_\beta = \tee[\preceq \psi(\beta)]\setminus \bigcup_{\gamma<\beta}\tee[\preceq \psi(\gamma)]$. The induced well-order $\leq$ of $\tee$ is defined by declaring $s\leq t$, where $s\in A_\beta$ and $t\in A_{\beta'}$, if $\beta<\beta'$ or if $\beta=\beta'$ and $s\preceq t$. Note that if $\tee$ is countable and $ht(\tee)\leq\omega$ then the well-ordering of $\tee$ induced as above by a surjection of $\omega$ onto $\tee$ is of order type $\omega$. In fact, the following statements are equivalent for an infinite tree $\tee$:
\begin{itemize}
\item[(i$_0$)] $\tee$ is countable and $\tee[\prec t]$ is finite for every $t\in\tee$;
\item[(ii$_0$)] $\tee$ is countable and $ht(\tee)\leq\omega$;
\item[(iii$_0$)] There exists a bijection $\tau$ of $\omega$ onto $\tee$ such that $\tau(l)\preceq \tau(m)$ implies $l\leq m$ for $l,m<\omega$.
\end{itemize}

\begin{example}\label{fullcounter}
Let $\Omega:=\bigcup_{n<\omega}\prod_n\omega$. That is, $\Omega$ is the set of all finite (including possibly empty) sequences of finite ordinals. We define an order $\sqsubseteq$ on $\Omega$ by saying that $s\sqsubseteq t$ if and only if $s$ is an initial segment of $t$. Note that $\Omega$ is a rooted tree, with its root being the empty sequence $\emptyset$. For $n<\omega$ and $t\in\Omega$ we denote by $n^\smallfrown t$ the concatenation of $(n)$ with $t$; that is, $n^\smallfrown t=(n)$ if $t=\emptyset$ and $n^\smallfrown t=(n,n_1,\ldots,n_k)$ if $t=(n_1,\ldots,n_k)$. It is straightforward to show that for an arbitrary tree $(\tee,\preceq)$ the following statements are equivalent to statements (i$_0$)-(iii$_0$) above and to each other:
\begin{itemize}
\item[(iv$_0$)] $\tee$ is order-isomorphic to a subtree of $\Omega$;
\item[(v$_0$)] $\tee$ is order-isomorphic to a downwards-closed subtree of $\Omega^\star$.
\end{itemize}
Moreover, if $\tee$ is rooted then (i$_0$)-(iv$_0$) are equivalent to:
\begin{itemize}
\item[(vi$_0$)] $\tee$ is order-isomorphic to a downwards-closed subtree of $\Omega$.
\end{itemize}
\end{example}

We now describe a method for inductively defining a decreasing (with respect to set inclusion) family of downwards closed subtrees of a given tree, indexed by the ordinals. To this end for a tree $(\tee,\preceq)$ let \begin{align*}\tee^{(0)}&=\tee;\\ \tee^{(\xi+1)}&=\tee^{(\xi)}\setminus \MAX(\tee^{(\xi)})\quad \mbox{for every ordinal }\xi;\mbox{ and,}\\ \tee^{(\xi)}&=\bigcap_{\xi'<\xi}\tee^{(\xi')}\quad\mbox{ if }\xi\mbox{ is a limit ordinal.}\end{align*}The fact that $\tee^{(\xi)}$ is downwards closed in $\tee$ for all ordinals $\xi$ follows from a straightforward transfinite induction. 

The \textit{rank} of a node $t\in\tee$ is defined to be the unique ordinal $\rho_\tee(t)$ such that $t\in \tee^{(\rho_T(t))}\setminus \tee^{(\rho_T(t)+1)}$, if it exists. Notice that if $s,t\in\tee$ are such that $t\prec s$ and $\rho_\tee(t)$ exists, then $\rho_\tee(s)$ exists and satisfies $\rho_\tee(s)< \rho_\tee(t)$ since the derived trees $\tee^{(\xi)}$ are downwards closed. It follows that if $t\in\tee$ is such that $\rho_\tee(s)$ exists for all $s\in\tee[t+]$, then $\rho_\tee(t)$ exists and satisfies 
\begin{equation}\label{rankeqn}
\rho_\tee(t)=\sup\{\rho_\tee(s)+1\mid s\in\tee[t+]\}\,.
\end{equation} 
Thus, if $t_0\in\tee$ is such that $\rho_\tee(t_0)$ does not exist, there exists $t_1\in \tee[t_0+]$ such that $\rho_\tee(t_1)$ does not exist; similarly, there exists $t_2\in \tee[t_1+]$ such that $\rho_\tee(t_2)$ does not exist, and in this way we inductively define an infinite chain $(t_n)_{n<\omega}$ in $\tee$, hence $\tee$ is not well-founded. Conversely, if $\tee$ contains an infinite branch, $\bee$ say, then, with $s_n$ denoting the element of $\bee$ of height $n$ in $\tee$, we have by induction that $\{s_n\mid n<\omega\}\subseteq\tee^{(\xi)}$ for all ordinals $\xi$, hence $\rho_\tee(s_n)$ is undefined for all $n$. We deduce that $\rho_\tee(t)$ exists for all $t\in\tee$ if and only if $\tee$ is well-founded, if and only if $\tee^{(\xi)}=\emptyset$ for some ordinal $\xi$. If $\tee$ is well-founded, the \textit{rank of} $\tee$ is the ordinal
\[
\rho(\tee):= \min\{ \xi\mid \tee^{(\xi)}=\emptyset\} =\sup\{ \rho_\tee(t)+1\mid t\in \tee\}.
\] Notice that if $\tee$ is rooted, with root $t_0$, say, then $\tee^{(\rho_\tee(t_0))}=\{ t_0\}$, hence $\rho(\tee)= \rho_\tee(t_0)+1$ is a successor ordinal.

We now give the definition of a \textit{blossomed} tree, due to \textcite{Gasparis2005}.

\begin{definition}\label{gaspardefn}
We say that a countable tree $(\tee,\preceq)$ is \textit{blossomed} if it is rooted, well-founded, and for each $t\in\tee\setminus\MAX(\tee)$ there exists a bijection $\psi:\omega \longrightarrow \tee[t+]$ such that $m\leq n<\omega $ implies $\rho_\tee(\psi(m))\leq \rho_\tee(\psi(n))$.
\end{definition}

Blossomed trees were used by Gasparis\autocite{Gasparis2005} to study fixing properties of operators of large Szlenk index acting on $C(K)$ spaces. The important property of a blossomed tree $\tee$ in studying the Szlenk index is that for every $t\in\tee\setminus\MAX(\tee)$ and cofinite subset $\queue\subseteq \tee[t+]$ we have 
\begin{equation}\label{jupiter}
\sup\{ \rho_\tee(t')\mid t'\in \queue \} = \sup\{\rho_\tee(t'')\mid t''\in \tee[t+] \};
\end{equation}
this condition clearly holds for any blossomed tree and, moreover, any countable, rooted, well-founded tree $\tee$ satisfying the property stated at \cref{jupiter} admits a subtree $\ess$ such that $\rho(\ess)=\rho(\tee)$ and $\ess$ is blossomed. Thus, blossomed trees can be thought of as the `minimal' trees satisfying these conditions and, moreover, the formally stronger definition of a blossomed tree is typically more convenient to work with than the property stated at \cref{jupiter} for the purposes of proving results concerning the Szlenk index, such as \cref{lbszlenk}.

The following example guarantees a rich supply of blossomed trees. Note that other examples of blossomed trees, namely the Schreier families of finite subsets of $\nat$, are used in the context of Szlenk indices in \textcite{Gasparis2005}. The construction of trees in \cref{blossexist} below is essentially the same as that given by Bourgain on p.91 of \textcite{Bourgain1979}.
\begin{example}\label{blossexist}
We construct, via transfinite induction on $\xi<\omega_1$, a family $(\tee_\xi)_{\xi<\omega_1}$ consisting of blossomed subtrees of $\Omega$ that satisfy $\rho(\tee_\xi)=\xi+1$ for each $\xi<\omega_1$. Set $\tee_0=\{ \emptyset \}$. Suppose $\xi>0$ is an ordinal such that the $\tee_\zeta$ has been defined for all $\zeta<\xi$; we define $\tee_\xi$ as follows. Let $(\xi_n)_{n=0}^\infty$ be a non-decreasing, cofinal sequence in $\xi$, and set \[ \tee_\xi = \{\emptyset\}\cup \{ n^\smallfrown t\mid n<\omega,\,t\in \tee_{\xi_n}\}\,. \] 
A straightforward transfinite induction on $\zeta\leq\xi$ shows that
\begin{equation}\label{auxilarious}
\forall\,\zeta\leq\xi\quad\tee_\xi^{(\zeta)}= \{ \emptyset\}\cup \bigcup_{n<\omega}\{n^\smallfrown t\mid t\in \tee_{\xi_n}^{(\zeta)} \}\,,
\end{equation}
hence
\begin{equation}\label{vauxilarious}
\forall\,\zeta<\xi\quad \MAX(\tee_\xi^{(\zeta)})= \bigcup_{n<\omega}\{n^\smallfrown t\mid t\in \MAX(\tee_{\xi_n}^{(\zeta)}) \}\,.
\end{equation}
Taking $\zeta=\xi$ in \cref{auxilarious} yields $\tee_\xi^{(\xi)}=\{\emptyset\}$, hence $\rho_{\tee_\xi}(\emptyset)=\xi$ and $\rho(\tee_\xi)=\xi+1$. 

For $n<\omega$ let $\imath_n :\tee_{\xi_n}\longrightarrow \tee_\xi$ be the map $t\mapsto n^\smallfrown t$. From \cref{vauxilarious} we have $\rho_{\tee_\xi}(n^\smallfrown t)=\rho_{\tee_{\xi_n}}(t)$ every $n<\omega$ and $t\in\tee_{\xi_n}$. Thus, if for $n<\omega$ and $t\in\tee_{\xi_n}\setminus \MAX(\tee_{\xi_n})$ the map $\psi:\omega\longrightarrow \tee_{\xi_n}[t+]$ is a bijection such that $(\rho_{\tee_{\xi_n}}(\psi(m)))_{m=0}^\infty$ is non-decreasing (as per \cref{blossexist}), then $\imath_n\circ\psi:\omega\longrightarrow \tee_\xi[(n^\smallfrown t)+]$ is a bijection with \[(\rho_{\tee_\xi}(\imath_n\circ \psi(m)))_{m=0}^\infty= (\rho_{\tee_\xi}(n^\smallfrown \psi(m)))_{m=0}^\infty= (\rho_{\tee_{\xi_n}}(\psi(m)))_{m=0}^\infty\] non-decreasing. Similarly, $n\mapsto (n)$ is a bijection of $\omega$ onto $\tee_\xi[\emptyset+]$ and \[ (\rho_{\tee_\xi}((n)))_{n=0}^\infty= (\rho_{\tee_{\xi_n}}(\emptyset))_{n=0}^\infty= (\xi_n)_{n=0}^\infty\] is non-decreasing. We have now shown that $\tee_\xi$ is blossomed, as required.
\end{example}

The following proposition collects properties of blossomed trees that we shall need in subsequent sections of the current paper. 
\begin{proposition}\label{blossrels}
Let $(\ess,\preceq')$ and $(\tee,\preceq)$ be countable, rooted, well-founded trees.
\begin{itemize} \item[(i)] If $\ess$ is blossomed and $\rho(\tee)\leq\rho(\ess)$ then $\tee$ is order isomorphic to a downwards closed subtree of $\ess$.
\item[(ii)] If $\ess$ is blossomed and $\ess'$ is a full subtree of $\ess$ then $\ess'$ is blossomed and $\rho(\ess')=\rho(\ess)$.
\end{itemize}
\end{proposition}
Assertion (i) of \cref{blossrels} is a trivial generalisation of Lemma~2.7 of \textcite{Gasparis2005}, requiring little change in the proof. Assertion (ii) is Lemma~2.8 of \textcite{Gasparis2005}. We refer the reader to \textcite{Gasparis2005} for the proofs.

There are various natural topologies for trees, many of which are described in \textcite{Nyikos1997}. The tree topology which will be of interest to us in \cref{kneesbees} is the \textit{coarse wedge topology}, which is compact and Hausdorff for many trees. The coarse wedge topology of $(\tee,\preceq)$ is that topology on $\tee$ formed by taking as a subbase all sets of the form $\tee[t\preceq]$ and $\tee\setminus \tee[t\preceq]$, where the order type of $\tee[\prec t]$ is either $0$ or a successor ordinal. For a tree $(\tee,\preceq)$, $t\in \tee$ and $\eff\subseteq \tee$, define \[ \dubyoo_\tee(t,\eff):= \tee[t\preceq]\setminus \bigcup_{s\in \eff}\tee[s\preceq] \,.\] The following proposition is clear.
\begin{proposition}\label{coarsefacts}
Let $(\tee,\preceq)$ be a tree and let $t\in\tee$ be such that the order type of $\tee[\prec t]$ is $0$ or a successor ordinal. Then the coarse wedge topology of $\tee$ admits a local base of clopen sets at $t$ consisting of all sets of the form $\dubyoo_\tee(t,\eff)$, where $\eff\subseteq \tee[t+]$ is finite.
\end{proposition}

The following result is proved in the aforementioned paper of Nyikos\autocite[][Corollary~3.5]{Nyikos1997}.
\begin{theorem}\label{coarsecompact}
Let $\tee$ be a tree. The following are equivalent.
\begin{itemize}
\item[(i)] $\tee$ is chain-complete and $\MIN(\tee)$ is finite.
\item[(ii)] The coarse wedge topology of $\tee$ is compact and Hausdorff.
\end{itemize}
\end{theorem}

We conclude the current subsection on trees with the following proposition.

\begin{proposition}\label{closedown}
Let $(\tee,\preceq)$ be a tree with $ht(\tee)\leq\omega$.
\begin{itemize} 
\item[(i)] Let $\ess\subseteq\tee$ be a downwards closed subset of $\tee$. Then $\ess$ is closed in the coarse-wedge topology of $\tee$.
\item[(ii)] Let $(\ess,\preceq')$ be a tree and suppose $\phi:\ess\longrightarrow\tee$ is an order-isomorphism of $\ess$ onto a downwards closed subset of $\tee$. Then $\phi$ is coarse wedge continuous.
\end{itemize}
\end{proposition}

\begin{proof}
We first prove (i). Suppose $t\in\tee\setminus\ess$. Then $\tee[t\preceq]$ is open in the coarse wedge topology of $\tee$ since $ht_\tee(t)<\omega$. Moreover $\ess\cap\tee[t\preceq]=\emptyset$ since $t\notin\ess$ and $\ess$ is downwards closed.

To prove (ii), first note that the sets $\tee[t\preceq]$ and $\tee\setminus\tee[t\preceq]$, where $t$ varies over all of $\tee$, form a subbasis of clopen sets for the coarse wedge topology of $\tee$. To establish the continuity of $\phi$ it therefore suffices to show that $\phi^{-1}(\tee[t\preceq])$ is clopen in $\ess$ for every $t\in\tee$. To this end suppose $t\in\tee$. If $t\notin\phi(\ess)$ then $\phi^{-1}(\tee[t\preceq])=\emptyset$ since $\phi(\ess)$ is downwards closed in $\tee$. On the other hand if $t\in\phi(\ess)$, say $t=\phi(s)$, then $\phi^{-1}(\tee[t\preceq])=\ess[s\preceq']$ since $\phi$ is an order isomorphism.
\end{proof}

\subsection{Operators on Banach spaces over trees}\label{jeefberky}

Let $(\tee,\preceq)$ be a tree. Define $\Sigma_\tee\colon\ell_1(\tee)\longrightarrow \ell_\infty(\tee)$ by setting
\[
\Sigma_\tee w = \Big(\sum_{s\preceq t}w(s)\Big)_{t\in\tee}\,,\qquad w\in \ell_1(\tee).
\]
Equivalently, $\Sigma_\tee$ is the unique element of $\allop(\ell_1(\tee),\ell_\infty(\tee))$ satisfying $\Sigma_\tee e_t = \chi_{\tee[t\preceq]}$ for each $t\in\tee$, with $\Vert \Sigma_\tee\Vert=1$ for nonempty $\tee$.

Notice that we can state some existing universality results in terms of operators of the form $\Sigma_\tee$. For instance, taking $\tee$ to be the set of natural numbers $\nat$ equipped with its usual order $\leq$, the operator $\Sigma_\tee$ is the aforementioned universal non-weakly compact operator of Lindenstrauss and Pe{\l}czy{\'n}ski (\vref{LPthm}). Moreover, taking $\tee$ to instead be the set of natural numbers $\nat$ equipped with the trivial order $=$ yields $\Sigma_\tee$ as the formal identity operator from $\ell_1$ to $\ell_\infty$, shown by Johnson to be universal for the class of non-compact operators (\vref{Jthm}). Amongst the outcomes of the current paper is that we add to the collection of trees $(\tee,\preceq)$ for which the corresponding operator $\Sigma_\tee$ is universal for the complement of some operator ideal.

We shall use the following proposition to determine whether $\Sigma_\tee$ factors through $T$, for certain trees $(\tee,\preceq)$ and operators $T$.

\begin{proposition}\label{factorkar}
Let $(\tee,\preceq)$ be a tree, $X$ and $Y$ Banach spaces and $T\in\allop(X,Y)$. The following are equivalent:
\begin{itemize}
\item[(i)] $\Sigma_\tee$ factors through $T$.
\item[(ii)] There exist $\delta>0$ and families $(x_t)_{t\in\tee}\subseteq B_X$ and $(x_t^\ast)_{t\in\tee}\subseteq T^\ast B_{Y^\ast}$ such that
\begin{equation}\label{deltadelta}
\langle x_t^\ast,x_s\rangle = \begin{cases}
\langle x_s^\ast,x_s\rangle\geq\delta&\mbox{if }s\preceq t\\
0&\mbox{if }s\npreceq t
\end{cases},
\quad\quad s,t\in\tee.
\end{equation}
\end{itemize}
\end{proposition}

\begin{proof}
First suppose that (i) holds and that $U\in\allop(\ell_1(\tee),X)$ and $V\in\allop(Y,\ell_\infty(\tee))$ are such that $VTU=\Sigma_\tee$. For each $t\in\tee$ let $f_t^\ast$ be the element of $\ell_\infty(\tee)^\ast$ satisfying $\langle f_t^\ast, z\rangle = \Vert V\Vert^{-1}z(t)$ for every $z\in \ell_\infty(\tee)$. For each $s,t\in\tee$ let $x_s=\Vert Ue_s\Vert^{-1}Ue_s\in B_X$ and $x_t^\ast= T^\ast V^\ast f_t^\ast\in T^\ast B_{Y^\ast}$. Then for $s,t\in\tee$ we have
\begin{align*}
\langle x_t^\ast,x_s\rangle &= \langle T^\ast V^\ast f_t^\ast,\Vert Ue_s\Vert^{-1}Ue_s\rangle\\
&=\Vert Ue_s\Vert^{-1}\langle f_t^\ast,VTUe_s\rangle\\
&=\Vert (Ue_s\Vert^{-1}\langle f_t^\ast,\Sigma_\tee e_s\rangle\\
&=\begin{cases}
\Vert Ue_s\Vert^{-1}\Vert V\Vert^{-1}\geq (\Vert U\Vert\Vert V\Vert)^{-1}&\mbox{if }s\preceq t\\
0&\mbox{if }s\npreceq t
\end{cases},
\end{align*}
hence (i) implies (ii).

Now suppose that (ii) holds and let $\delta>0$, $(x_t)_{t\in\tee}\subseteq B_X$ and $(x_t^\ast)_{t\in\tee}\subseteq T^\ast B_{Y^\ast}$ be such that \cref{deltadelta} holds. Define $U\colon \ell_1(\tee)\longrightarrow X$ by setting 
\[
Uw=\sum_{t\in\tee}\frac{w(t)}{\langle x_t^\ast,x_t\rangle}x_t
\]
for each $w\in \ell_1(\tee)$. Then $U$ is well-defined, linear and continuous with $\Vert U\Vert\leq \delta^{-1}$. For each $t\in\tee$ choose $v_t^\ast\in B_{Y^\ast}$ such that $T^\ast v_t^\ast=x_t^\ast$. The map $V: Y\longrightarrow \ell_\infty(\tee)$ given by setting $Vy=(\langle v_t^\ast,y\rangle)_{t\in\tee}$ for each $y\in Y$ is well-defined, linear and continuous with $\Vert V\Vert\leq 1$. To complete the proof we show that $VTU=\Sigma_\tee$. 
To this end note that for $s\in\tee$ we have \[ VTUe_s= VT(\langle x_s^\ast, x_s\rangle^{-1}x_s )= (\langle x_s^\ast, x_s\rangle^{-1}\langle v_t^\ast ,Tx_s \rangle)_{t\in\tee} = (\langle x_s^\ast, x_s\rangle^{-1}\langle x_t^\ast ,x_s \rangle)_{t\in\tee} = \chi_{\tee[s\preceq]},\] hence $VTU=\Sigma_\tee$.
\end{proof}

The following result may be proved by an appeal to \cref{factorkar}, but an equally easy direct proof is possible (we omit the details).

\begin{proposition}\label{subtreefac}
Let $(\ess,\preceq')$ and $(\tee,\preceq)$ be trees and suppose that $\ess$ is order isomorphic to a subtree of $\tee$. Then $\Sigma_\ess$ factors through $\Sigma_\tee$.
\end{proposition}

Notice that if $\opideal$ and $\scrjay$ are operator ideals and $T\in \scrjay\cap\complement\opideal$ is universal for $\complement\opideal$, then every universal element of $\complement\opideal$ belongs to $\scrjay$. In particular, it is a consequence of the following proposition that if $\opideal$ is an operator ideal and $\tee$ is a tree such that $\Sigma_\tee$ is universal for $\complement\opideal$, then any operator that is universal for $\complement\opideal$ must be strictly singular.

\begin{proposition}\label{otherclasses}
Let $(\tee,\preceq)$ be a tree. Then:
\begin{itemize}
\item[(i)] $\Sigma_\tee$ is strictly singular.
\item[(ii)] $\Sigma_\tee$ is weakly compact if and only if $\tee$ is well-founded.
\item[(iii)] $\Sigma_\tee$ is compact if and only if $\tee$ is finite, if and only if $\Sigma_\tee$ is finite rank.
\end{itemize}
\end{proposition}

Suppose $(\tee,\preceq)$ is a tree. For the purposes of proving \cref{otherclasses} we now recall the definition of the \textit{James tree space of $\tee$}, denoted $J(\tee)$, which is the completion of $c_{00}(\tee)$ with respect to the norm $\Vert \cdot \Vert_{J(\tee)}$ on $c_{00}(\tee)$ that is defined by setting
\[
\Vert x\Vert_{J(\tee)} = \sup\bigg\{ \Big( \sum_{i=1}^k\big\vert \sum_{t\in \ess_i}x(t)\big\vert^2 \Big)^{1/2}\mid   \ess_1, \ldots, \ess_k\subseteq \tee \mbox{ pairwise disjoint intervals}\bigg\}
\]
for each $x\in c_{00}(\tee)$. The formal identity map from $(c_{00}(\tee), \Vert \cdot\Vert_{\ell_1(\tee)})$ to $(c_{00}(\tee),\Vert \cdot\Vert_{J(\tee)})$ is norm continuous and therefore admits a (unique) continuous linear extension $A_\tee\in \allop(\ell_1(\tee), J_2(\tee))$. Moreover, the linear map $x\mapsto ( \sum_{s\preceq t}x(s) )_{t\in\tee}$ from $c_{00}(\tee)$ to $\ell_\infty(\tee)$ is norm continuous with with respect to the norm $\Vert \cdot \Vert_{J(\tee)}$ on $c_{00}(\tee)$, and thus extends (uniquely) to some $B_\tee\in\allop(J(\tee),\ell_\infty(\tee))$. Since $\Sigma_\tee = B_\tee A_\tee$ we have that $\Sigma_\tee$ factors through the James tree space $J(\tee)$.

\begin{remark}
For an introduction to properties of James tree spaces we refer the interested reader to \textcite{Brackebusch1988}. 
\end{remark}

\begin{proof}[Proof of \cref{otherclasses}]
Assertion (i) of the proposition follows from the fact that the domain of $\Sigma_\tee$, namely $\ell_1(\tee)$, is $\ell_1$-saturated, whilst $\Sigma_\tee$ has already been seen above to factor through the $\ell_2$-saturated space $J(\tee)$. (See Lemma~2 and the final remark of \textcite{Hagler1978} for details of the proof that $J(\tee)$ is $\ell_2$-saturated.)

For (ii), the assertion that $\Sigma_\tee$ is weakly compact whenever $\tee$ is well-founded follows from the aforementioned fact that $\Sigma_\tee$ factors through the James tree space $J(\tee)$ of $\tee$ and the fact that $J(\tee)$ is reflexive if and only if $\tee$ is well-founded. The proof of this latter fact is obtained via a straightforward transfinite induction on $\rho(\tee)$, using the following facts: an $	\ell_2$-direct sum of a family of reflexive spaces is reflexive; and, for a rooted tree $\tee$, the Banach space $\big(\bigoplus_{t\in \MIN(\tee^\star)}J(\tee[t\preceq])\big)_{\ell_2}$ is isometrically isomorphic to a codimension 1 subspace of $J(\tee)$; the remaining details are omitted. On the other hand, if $\tee$ is not well-founded then $\tee$ contains a path order-isomorphic to $\nat$ equipped with its usual order $\leq$. It follows then by \cref{subtreefac} that $\Sigma_\tee$ factors the universal non-weakly compact operator of Lindenstrauss and Pe{\l}czy{\'n}ski (\cref{LPthm}), hence $\Sigma_\tee$ fails to be weakly compact whenever $\tee$ is not well founded.

To prove (iii), first note that if $\tee$ is finite then the codomain is the finite-dimensional space $\ell_\infty(\tee)$, hence $\Sigma_\tee$ is finite rank and therefore compact. Conversely, if $\tee$ is infinite then the set $\{ \Sigma_\tee e_t\mid t\in\tee\}$ is an infinite $1$-separated subset of $\Sigma_\tee B_{\ell_1(\tee)}$, hence in this case $\Sigma_\tee$ is non-compact, hence non-finite rank.
\end{proof}

We now establish a connection between the rank $\rho(\tee)$ and the Szlenk indices of $\Sigma_\tee$ in the particular case that $\tee$ is blossomed. The proof of the following proposition is similar to the last part of the proof Proposition~6.2 of \textcite{Beanland2018}.

\begin{proposition}\label{lbszlenk}
Let $(\tee,\preceq)$ be a blossomed tree and $\epsilon\in(0,1)$. Then $Sz(\Sigma_{\tee^\star},\epsilon)\geq \rho(\tee)$.
\end{proposition}

\begin{proof}
Fix $\epsilon \in(0,1)$, let $t_\emptyset$ denote the root of $\tee$ and let $B=\Sigma_{\tee^\star}^\ast B_{\ell_\infty(\tee^\star)^\ast}$. For each $t\in\tee^\star$ let $f_t^\ast\in\ell_\infty(\tee^\star)^\ast$ be the evaluation functional of $\ell_\infty(\tee^\star)$ at $t$ and let $g_t^\ast:= \Sigma_{\tee^\star}^\ast f_t^\ast \in B$. Let $g_{t_\emptyset}^\ast:=0\in B$. Note that $\{ g_t^\ast\mid t\in\tee\}$ is $\epsilon$-separated, for if $s,t\in\tee$ are such that $s\notin\tee[t\preceq]$ then $\Vert g_t^\ast-g_s^\ast\Vert \geq \langle g_t^\ast-g_s^\ast,e_t\rangle=1>\epsilon$. For each $t\in\tee\setminus\MAX(\tee)$ let $(t_m)_{m=0}^\infty$ be an enumeration of $\tee[t+]$ with $(\rho_\tee(t_m))_{m=0}^\infty$ non-decreasing. For every $x\in\ell_1(\tee^\star)$ and $t\in\tee\setminus\MAX(\tee)$ we have $\langle g_{t_m}^\ast,x\rangle = \sum_{s\preceq t_m}x(s)\rightarrow \sum_{s\preceq t}x(s)=\langle g_t^\ast,x\rangle$ as $m\rightarrow \infty$, so that $(g_{t_m}^\ast)_{m=0}^\infty$ is weak$^\ast$-convergent to $g_t^\ast$.

We will show by transfinite induction that 
\begin{equation}\label{nearlycrawling}
\{ g_t^\ast\mid t\in\tee^{(\xi)}\setminus\tee^{(\xi+1)}\}\subseteq s_\epsilon^\xi(B).
\end{equation}
for every $\xi<\rho(\tee)$. \cref{nearlycrawling} is trivially true for $\xi=0$. Suppose $\zeta\in(0,\rho(\tee))$ is such that \cref{nearlycrawling} holds for every $\xi<\zeta$; to complete the induction we show that $\{ g_t^\ast\mid t\in\tee^{(\zeta)}\setminus\tee^{(\zeta+1)}\}\subseteq s_\epsilon^\zeta(B)$. To this end suppose $t\in\tee^{(\zeta)}\setminus\tee^{(\zeta+1)}$. By the inductive hypothesis and the fact that $(\rho_\tee(t_m))_{m=1}^\infty$ is non-decreasing we have that $(g_{t_l}^\ast)_{l=m}^\infty\subseteq s_\epsilon^{\rho_\tee(t_m)}(B)$. Thus $g_t^\ast \in s_\epsilon^{\rho_\tee(t_m)}(B)$ as the latter set is weak$^\ast$-closed. This yields $g_t^\ast \in s_\epsilon^{\rho_\tee(t_m)+1}(B)$ since $\Vert g_t^\ast-g_{t_l}^\ast\Vert>\epsilon$ for all $l\in [m,\omega)$. It follows that $g_t^\ast\in\bigcap_{m<\omega}s_\epsilon^{\rho_\tee(t_m)+1}(B)=s_\epsilon^{\rho_\tee(t)}(B)=s_\epsilon^\zeta(B)$, 
as required. With the induction now complete, taking $\xi=\rho(\tee)-1$ in \cref{nearlycrawling} yields $g_{t_\emptyset}^\ast \in s_\epsilon^{\rho(\tee)-1}(B)$, from which the proposition follows.
\end{proof}

\section{Absolutely convex sets of large Szlenk index}\label{mainlemmasection}
This section is devoted to proving our key technical result, \cref{oopsblah}, from which the main results of the paper (stated in the Introduction) and a number of other results in subsequent sections of the paper are derived. In order to state \cref{oopsblah} we introduce the following notation. For a family $\mathfrak{T}=((\tee_n,\preceq_n))_{n<\omega} $, where each $(\tee_n,\preceq_n)$ is a rooted tree, define $\llbracket\mathfrak{T}\rrbracket := \{ \emptyset\} \cup \bigcup_{n<\omega}(\{ n\}\times \tee_n^\star)$, so that $\llbracket\mathfrak{T}\rrbracket$ is a rooted tree when equipped with the order $\preceq_\mathfrak{T}$ on $\llbracket\mathfrak{T}\rrbracket$ defined by setting $\emptyset \preceq_\mathfrak{T} t$ for all $t\in \llbracket\mathfrak{T}\rrbracket$ and $(n_1,t_1)\preceq_\mathfrak{T} (n_2,t_2)$ if and only if $n_1=n_2$ and $t_1\preceq_{n_1}t_2$.

\begin{theorem}\label{oopsblah}
Let $X$ be a Banach space, $K\subseteq X^\ast$ a non-empty, absolutely convex, weak$^\ast$-compact set, $\delta,\theta>0$ positive real numbers, $(\epsilon_n)_{n<\omega}$ a family of positive real numbers, $(\xi_n)_{n<\omega}$ a family of countable ordinals such that $s_{\epsilon_n}^{\xi_n}(K)\neq \emptyset$ for all $n<\omega$, and $\mathfrak{T}=( (\tee_n,\preceq_n))_{n<\omega}$ a family of countable, well-founded, rooted trees such that $\rho(\tee_n)\leq\xi_n+1$ for all $n<\omega$.
\begin{itemize}
\item[(i)] There exist families $(x_t^\ast)_{t\in\llbracket\mathfrak{T}\rrbracket^\star}\subseteq K$ and $(x_t)_{t\in\llbracket\mathfrak{T}\rrbracket^\star}\subseteq S_X$ such that 
\begin{equation}\label{biorthogger}
\langle x_t^\ast ,x_s\rangle=\begin{cases}
\langle x_s^\ast,x_s\rangle>\frac{\epsilon_n}{8+\theta}&\mbox{if }s\preceq_\mathfrak{T} t\in \{n\}\times \tee_n^\star\\ 0 &\mbox{if }s\npreceq_\mathfrak{T} t
\end{cases}, \qquad s,t\in \llbracket\mathfrak{T}\rrbracket^\star,\,n<\omega.
\end{equation}
Moreover, for any bijection $\tau:\omega\longrightarrow \llbracket\mathfrak{T}\rrbracket$ such that $\tau(l)\preceq_\mathfrak{T} \tau(m)$ implies $l\leq m$ we may choose $(x_t)_{t\in\llbracket\mathfrak{T}\rrbracket^\star}$ so that $(x_{\tau(m)})_{m=1}^\infty$ is a basic sequence with basis constant not exceeding $1+\delta$.
\item[(ii)] If $X^\ast$ is norm separable then the families $(x_t^\ast)_{t\in\llbracket\mathfrak{T}\rrbracket^\star}$ and $(x_t)_{t\in\llbracket\mathfrak{T}\rrbracket^\star}$ in (i) may be chosen so that $(x_{\tau(m)})_{m=1}^\infty$ is shrinking and $(Qx_{\tau(m)})_{m=1}^\infty$ is a shrinking basis for $X/Z$ with basis constant not exceeding $1+\delta$. Here $Z=\bigcap_{t\in\llbracket\mathfrak{T}\rrbracket^\star}\ker(x_t^\ast)$ and $Q:X\longrightarrow X/Z$ is the quotient map. 
\end{itemize}
\end{theorem}

The proof of \cref{oopsblah} shall invoke the following lemma due to G.~Lancien\autocite[][p.67]{Lancien1996}, who established the result for the special case $K=B_{X^\ast}$ and $\zeta$ of the form $\omega^\alpha$ for some ordinal $\alpha$; the same argument gives the more general statement presented below.
\begin{lemma}\label{obstaclecentre}
Let $X$ be a Banach space, $K\subseteq X^\ast$ an absolutely convex, weak$^\ast$-compact set, $\zeta$ an ordinal and $\epsilon>0$. 
If $s_\epsilon^\zeta(K)\neq\emptyset$ then
\begin{equation}\label{dingraniel}
\forall \,n<\omega\quad 0\in s_{\epsilon/2^{n+1}}^{\zeta 2^n}(K).
\end{equation}
\end{lemma}

We require the following result, which is Lemma~2.2 of \textcite{Dilworth2017}.
\begin{lemma}\label{lowerlemma}
Let $X$ be a Banach space, $\nu>0$ a real number, $F$ a finite-dimensional subspace of $X^\ast$, $A$ a $\frac{\nu}{4+2\nu}$-net in $S_F$ and $\{ y_{f^\ast}\mid f^\ast\in A\}\subseteq S_X$ a family such that $\inf\{ \vert f^\ast(y_{f^\ast})\vert\mid f^\ast\in A\}\geq \frac{4+\nu}{4+2\nu}$. Then for every $x^\ast\in \{ y_{f^\ast}\mid f^\ast\in A\}^\perp$ we have $ \sup\{ \vert x^\ast (y)\vert\mid y\in B_{F_\perp} \}\geq \frac{1}{2+\nu}\Vert x^\ast\Vert$.
\end{lemma}

We do not know a reference for the following result, so we provide the straightforward proof.
\begin{lemma}\label{zippinbase}
Let $X$ be a Banach space and $x^\ast\in X^\ast$. Then \[ \{ x^\ast + \epsilon B_{X^\ast}^\circ + C^\perp \mid \epsilon>0,\, C\subseteq X, \, \vert C\vert <\infty\} \] is a local base for the weak$^\ast$ topology of $X^\ast$ at $x^\ast$.
\end{lemma}

\begin{proof}
We assume $x^\ast=0$, from which the general case follows easily. For $\epsilon>0$ and finite $C\subseteq X$ we have $\epsilon B_{X^\ast}^\circ+C^\perp\subseteq \bigcap_{x\in C}\{ y^\ast \in X^\ast \mid \vert y^\ast (x)\vert <\epsilon\}$, so it suffices to show that $\epsilon B_{X^\ast}^\circ+C^\perp$ is weak$^\ast$-open for $\epsilon>0$ and finite $C\subseteq X$.

Fix $\epsilon>0$ and a finite set $C\subseteq X$. Set $Y=\spn(C)$. Notice that $C^\perp=Y^\perp$. We have
\[
\epsilon B_{X^\ast}^\circ +C^\perp = \{ x^\ast \in X^\ast\mid \dist(x^\ast,Y^\perp)<\epsilon\} = Q^{-1}(\epsilon B_{X^\ast/Y^\perp}^\circ)
\]
where $Q:X^\ast\longrightarrow X^\ast /Y^\perp$ is the canonical quotient map. It is thus enough to show that $Q$ is weak$^\ast$ continuous. This follows from the classic decomposition $Q=S\circ R$, where $R:X^\ast\longrightarrow Y^\ast$ is the restriction map, which is clearly weak$^\ast$ continuous, and $S:Y^\perp \longrightarrow X^\ast /Y^\perp$, defined by $S(Rx^\ast)=Qx^\ast$, is weak$^\ast$ continuous as a bounded linear map between finite dimension normed spaces.
\end{proof}

\begin{lemma}\label{turnbeak}
Let $X$ be a Banach space, $K$ and $L$ weak$^\ast$-compact subsets of $X^\ast$, $\xi$ an ordinal and $\epsilon>0$. If $x^\ast\in s_\epsilon^\xi(K)$ and $y^\ast\in L$ then $x^\ast+y^\ast\in s_\epsilon^\xi(K+L)$.
\end{lemma}

\begin{proof}
We proceed by transfinite induction on $\xi$. The assertion of the lemma is true for $\xi=0$. Suppose that $\zeta>0$ is an ordinal such that the assertion of the lemma is true for $\xi=\zeta$; we will show that it is true then for $\xi=\zeta+1$. Let $x^\ast\in s_\epsilon^{\zeta+1}(K)$ and $y^\ast\in L$. Since $x^\ast\in s_\epsilon^{\zeta}(K)$ it follows from the induction hypothesis that $x^\ast+y^\ast \in s_\epsilon^\zeta(K+L)$. Since $x^\ast\in s_\epsilon^{\zeta+1}(K)$ we have that for any weak$^\ast$-neighbourhood $\yoo$ of $x^\ast+y^\ast$ there exist $x_1^\ast, x_2^\ast\in s_\epsilon^\zeta(K)\cap (-y^\ast+\yoo)$ such that $\Vert x_1^\ast-x_2^\ast\Vert>\epsilon$. Since $x_1^\ast+y^\ast, x_2^\ast+y^\ast\in\yoo\cap (K+L)$ and $\Vert (x_1^\ast+y^\ast)-(x_2^\ast+y^\ast)\Vert >\epsilon$, we deduce that $x^\ast+y^\ast\in s_\epsilon^{\zeta+1}(K+L)$. 

Now suppose $\zeta$ is a limit ordinal and the assertion of the lemma is true for all $\xi<\zeta$. For $x^\ast\in s_\epsilon^\zeta (K)= \bigcap_{\xi<\zeta}s_\epsilon^\xi(K)$ and $y^\ast \in L$ we have
\[
x^\ast+y^\ast\in \bigcap_{\xi<\zeta}s_\epsilon^\xi(K+L)= s_\epsilon^\zeta(K+L),
\]
which completes the induction. 
\end{proof}

The final preliminary result before proving \cref{oopsblah} is the following theorem due to Kadec\autocite{Kadec1958} and Klee\autocite{Klee1960}; a short proof due to Davis and Johnson (sketched in \textcite{Davis1973}) has been published by Lindenstrauss and Tzafriri\autocite[][p.13]{Lindenstrauss1977}.

\begin{theorem}\label{kkthm}
Let $(X,\Vert \cdot\Vert)$ be a Banach space and $c>1$ a real number. If $X^\ast$ is norm separable then $X$ admits a norm $\tripnorm \cdot\tripnorm$ such that the following properties hold:
\begin{itemize}
\item[(i)] For every $x^\ast\in X^\ast$ we have $\Vert x^\ast \Vert \leq \tripnorm x^\ast\tripnorm \leq c\Vert x^\ast\Vert$.
\item[(ii)] If $(x_n^\ast)\subseteq X^\ast$ and $x^\ast\in X^\ast$ are such that $x_n^\ast\stackrel{w^\ast}{\rightarrow}x^\ast$ and $\tripnorm x_n^\ast\tripnorm \rightarrow\tripnorm x^\ast\tripnorm$, then $\tripnorm x_n^\ast- x^\ast\tripnorm\rightarrow0$.
\end{itemize}
\end{theorem}

As is implicit in the statement of \cref{kkthm}, when we apply the renorming result \cref{kkthm} we shall use $\tripnorm \cdot\tripnorm$ to denote also the corresponding induced norm on duals, subspaces, quotients and operators on $X$.

\begin{proof}[Proof of \cref{oopsblah}]
We shall first prove (i) without any assumption on the norm density of $X^\ast$, then show how to modify the arguments in the proof of (i) to obtain also the assertions of (ii) when $X^\ast$ is assumed to be norm separable.

Fix $\delta,\theta>0$ and let $\nu>0$ be a real number small enough that \begin{equation}\label{festyesty}
\frac{1-2\nu}{4(2+\nu)(1+\nu)} \geq\frac{1}{8+\theta}.
\end{equation} Let $\tau: \omega\longrightarrow \llbracket\mathfrak{T}\rrbracket$ be a bijection such that $\tau(l)\preceq_\mathfrak{T}\tau(m)$ implies $l\leq m$ (c.f. the paragraph preceding \cref{fullcounter}), noting that we necessarily have $\tau(0)=\emptyset$, the root of $\llbracket\mathfrak{T}\rrbracket$. For $0<m<\omega$ we may write $\tau(m)=(n_m,t_m)$, where $n_m <\omega$ and $t_m\in \tee_{n_m}^\star$. Fix a sequence $(\delta_m)_{m=1}^\infty\subseteq (0,1)$ of positive real numbers such that $\sum_{m=1}^\infty\delta_m<\infty$ and $\prod_{m=1}^\infty (1-\delta_m)\geq (1+\delta)^{-1}$. Proceeding via an induction over $m\in[1,\omega)$, we shall construct families $(f_{\tau(m)}^\ast)_{m<\omega}\subseteq X^\ast$ and $(x_{\tau(m)})_{1\leq m<\omega}\subseteq S_X$ satisfying the following conditions for all $m\in[1,\omega)$: 
\begin{itemize}
\item[(I)] $f_{\tau(m)}^\ast \in s_{\epsilon_{n_m}/2}^{\rho_{\tee_{n_m}}(t_m)}\big((1+\nu\sum_{j=1}^m 2^{-j})K\big)$;
\item[(II)] For all $i,j\in [1,m]$,
\begin{equation}\label{nutellaman}
\langle f_{\tau(j)}^\ast ,x_{\tau(i)}\rangle=\begin{cases}
\langle f_{\tau(i)}^\ast,x_{\tau(i)}\rangle >\frac{(1-2\nu)\epsilon_{n_i}}{4(2+\nu)} &\mbox{if }{\tau(i)}\preceq_\mathfrak{T} {\tau(j)}\\ 0 &\mbox{if }{\tau(i)}\npreceq_\mathfrak{T} {\tau(j)}
\end{cases}; \mbox{ and,}
\end{equation}
\item[(III)] For all $x\in \spn\{ x_{\tau(l)}\mid 1\leq l< m\}$ and scalars $a$ we have $\Vert x+ a x_{\tau(m)}\Vert \geq (1-\delta_m)\Vert x\Vert$.
\end{itemize}

Since for $0<m<\omega$ we have $s_{\epsilon_{n_m}/2}^{\rho_{\tee_{n_m}}(t_m)}\big((1+\nu\sum_{j=1}^m 2^{-j})K\big)\subseteq (1+\nu)K$, once the induction is complete the first assertion of (i) then follows from \cref{festyesty}, (I) and (II) by taking $x_t^\ast = \frac{1}{1+\nu}f_t^\ast$ for each $t\in\llbracket\mathfrak{T}\rrbracket^\star$.

The second assertion of (i) follows from the Grunblum criterion\autocite[See e.g.][Proposition~1.1.9]{Albiac2016} and the fact that, by (III), for $1\leq l< m<\omega$ and scalars $a_1,\ldots,a_m$ we have
\[
\Big\Vert \sum_{q=1}^la_qx_{\tau(q)}\Big\Vert \leq \frac{1}{\displaystyle\prod_{q=l+1}^m(1-\delta_q)}\Big\Vert \sum_{q=1}^ma_qx_{\tau(q)}\Big\Vert
\leq (1+\delta)\Big\Vert \sum_{q=1}^ma_qx_{\tau(q)}\Big\Vert\,.
\] 

For each $n<\omega$ let $o_n$ denote the root of $\tee_n$. It follows from \cref{dingraniel} that 
\begin{equation}\label{inallsets}0\in \bigcap_{n<\omega}s_{\epsilon_n/2}^{\xi_n}(K)\subseteq \bigcap_{n<\omega}s_{\epsilon_n/2}^{\rho(\tee_n)-1}(K)= \bigcap_{n<\omega}s_{\epsilon_n/2}^{\rho_{\tee_n}(o_n)}(K).\end{equation}
Define $ f_{\tau(0)}^\ast:=0 \in \bigcap_{n<\omega}s_{\epsilon_n/2}^{\xi_n}(K)$. Since 
\[
f_{\tau(0)}^\ast \in s_{\epsilon_{n_1}/2}^{\rho_{\tee_{n_1}}(o_{n_1})}(K)\subseteq s_{\epsilon_{n_1}/2}^{\rho_{\tee_{n_1}}(t_1)+1}(K)=s_{\epsilon_{n_1}/2}\big(s_{\epsilon_{n_1}/2}^{\rho_{\tee_{n_1}}(t_1)}(K)\big),
\]
it follows from the definition of the derivation $s_{\epsilon_{n_1}/2}$ that there is $f_{\tau(1)}^\ast\in s_{\epsilon_{n_1}/2}^{\rho_{\tee_{n_1}}(t_1)}(K)$ such that $\Vert f_{\tau(1)}^\ast \Vert = \Vert f_{\tau(1)}^\ast - f_{\tau(0)}^\ast\Vert >\epsilon_{n_1}/4$. Choose $x_{\tau(1)}\in S_X$ so that \[ \langle f_{\tau(1)}^\ast, x_{\tau(1)}\rangle >\frac{\epsilon_{n_1}}{4}
.\] It is readily checked that (I)-(III) hold for $m=1$.

Fix $k\in [1,\omega)$ and suppose that the points $f_{\tau(m)}^\ast\in s_{\epsilon_{n_m}/2}^{\rho_{\tee_{n_m}}(t_m)}\big((1+\nu\sum_{j=1}^m 2^{-j})K\big)$ and $x_{\tau(m)}\in S_X$ have been defined for $1\leq m\leq k$ in such a way that properties (I)-(III) are satisfied for $1\leq m\leq k$. Let $(k+1)^-$ denote the unique ordinal less than $k+1$ and such that $\tau((k+1)^-)=\tau(k+1)^-$. To carry out the inductive step of the proof we show how to construct $f_{\tau(k+1)}^\ast\in s_{\epsilon_{n_{k+1}}/2}^{\rho_{\tee_{n_{k+1}}}(t_{k+1})}\big((1+\nu\sum_{j=1}^{k+1} 2^{-j})K\big)$ and $x_{\tau(k+1)}\in S_X$ so that (I)-(III) are satisfied for $1\leq m\leq k+1$. Our first task will be to define $f_{\tau(k+1)}^\ast$ as a point inside a certain weak$^\ast$-neighbourhood of $f_{\tau(k+1)^-}^\ast$ and then show that (I) holds for $m=k+1$. To this end let $G$ be a finite $\delta_{k+1}$-net in $S_{\spn\{ x_{\tau(i)}\mid 1\leq i\leq k \}}$ and for each $g\in G$ let $h_g^\ast\in X^\ast$ be such that $\langle h_g^\ast ,g\rangle=1$. Set $F= \spn\big( \{ f_{\tau(i)}^\ast\mid 1\leq i\leq k \} \cup \{ h_g^\ast \mid g\in G\} \big)\subseteq X^\ast$, let $A$ be a finite $\frac{\nu}{4+2\nu}$-net in $S_F$ and let $\{ y_{f^\ast}\mid f^\ast\in A \}\subseteq S_X$ be such that $f^\ast(y_{f^\ast})\geq \frac{4+\nu}{4+2\nu}$ for each $f^\ast\in A$. Let
\begin{align*}
\yoo_1&=\bigcap_{i=1}^k\Big\{ x^\ast\in X^\ast\,\,\Big\vert\,\, \vert \langle x^\ast - f_{\tau(k+1)^-}^\ast,x_{\tau(i)}\rangle\vert<\frac{2^{-k-4}\nu(1-2\nu)\epsilon_{n_i}}{k(2+\nu)(1+\nu)} \Big\};\mbox{ and,}\\
\yoo_2&= f_{\tau(k+1)^-}^\ast + \frac{\nu\epsilon_{n_{k+1}}}{4(2+\nu)}B_{X^\ast}^\circ+ \big(\{ x_{\tau(i)}\mid 1\leq i\leq k \}\cup\{ y_{f^\ast}\mid f^\ast\in A\}\big)^\perp.
\end{align*}
Note that $\yoo_2$ is a weak$^\ast$-open neighbourhood of $f_{\tau(k+1)^-}^\ast$ by \cref{zippinbase}, hence $\yoo:= \yoo_1\cap\yoo_2$ is a weak$^\ast$-open neighbourhood of $f_{\tau(k+1)^-}^\ast$. On the one hand, if $(k+1)^-=0$ then $f_{\tau(k+1)^-}^\ast=0$, hence by \cref{inallsets} in this case we have
\begin{align} f_{\tau(k+1)^-}^\ast=0\in s_{\epsilon_{n_{k+1}}/2}^{\rho_{\tee_{n_{k+1}}}(o_{n_{k+1}})}(K)&\subseteq s_{\epsilon_{n_{k+1}}/2}^{\rho_{\tee_{n_{k+1}}}(o_{n_{k+1}})}\big((1+\nu\sum_{j=1}^k2^{-j})K\big) \notag \\ &\subseteq s_{\epsilon_{n_{k+1}}/2}^{\rho_{\tee_{n_{k+1}}}(t_{k+1})+1}\big((1+\nu\sum_{j=1}^k2^{-j})K\big) \label{holdjohn}.
\end{align} 
On the other hand, if $(k+1)^-\neq0$ then $n_{(k+1)^-}=n_{k+1}$ and $t_{(k+1)^-}=t_{k+1}^-\in\tee_{n_{k+1}}$. So by the hypothesis that (I) holds for $1\leq m\leq k$, we have 
\begin{align}
f_{\tau(k+1)^-}^\ast\in s_{\epsilon_{n_{k+1}}/2}^{\rho_{\tee_{n_{k+1}}}(t_{k+1}^-)}\big((1+\nu\sum_{j=1}^{(k+1)^-}2^{-j})K\big)&\subseteq s_{\epsilon_{n_{k+1}}/2}^{\rho_{\tee_{n_{k+1}}}(t_{k+1}^-)}\big((1+\nu\sum_{j=1}^k2^{-j})K\big)\notag\\&\subseteq s_{\epsilon_{n_{k+1}}/2}^{\rho_{\tee_{n_{k+1}}}(t_{k+1})+1}\big((1+\nu\sum_{j=1}^k2^{-j})K\big).\label{sesbar}\end{align}

It follows from \cref{holdjohn,sesbar}, and the definition of the derivation $s_\epsilon^\xi$ for $\epsilon>0$ and $\xi\in\ord$ that there exists $u^\ast\in\yoo\cap s_{\epsilon_{n_{k+1}}/2}^{\rho_{\tee_{n_{k+1}}}(t_{k+1})}\big((1+\nu\sum_{j=1}^k2^{-j})K\big)$ such that $\Vert f_{\tau(k+1)^-}^\ast - u^\ast\Vert>\epsilon_{n_{k+1}}/4$. Define
\[
f_{\tau(k+1)}^\ast:= u^\ast -\sum_{l=1}^k  \frac{\langle u^\ast- f_{\tau(k+1)^-}^\ast, x_{\tau(l)} \rangle}{\langle f_{\tau(l)}^\ast,x_{\tau(l)} \rangle}(f_{\tau(l)}^\ast -f_{\tau(l)^-}^\ast).
\]

Since $u^\ast\in\yoo_1$ and since $f_{\tau(l)}^\ast - f_{\tau(l)^-}^\ast\in2(1+\nu)K$ for $1\leq l\leq k$ (by the assumption that (I) holds for $1\leq m\leq k$), it follows from the definition of $f_{\tau(k+1)}^\ast$ that we have $f_{\tau(k+1)}^\ast-u^\ast \in cK$, where $c>0$ is a scalar that may be taken to satisfy
\begin{align}
c&\leq
\sum_{l=1}^k\frac{\vert \langle u_{\tau(k+1)}^\ast-f_{\tau(k+1)^-}^\ast ,x_{\tau(l)}\rangle\vert}{\vert \langle f_{\tau(l)}^\ast ,x_{\tau(l)}\rangle\vert}2(1+\nu) \notag\\
&\leq \sum_{l=1}^k\frac{2^{-k-4}\nu(1-2\nu)\epsilon_{n_l}}{k(2+\nu)(1+\nu)} \frac{4(2+\nu)}{(1-2\nu)\epsilon_{n_l}}2(1+\nu) \notag\\\notag 
&=   \sum_{l=1}^k \frac{2^{-k-1}\nu}{k}  \\
&= \nu2^{-k-1}.\label{kimmingswap}
\end{align}
An appeal to \cref{turnbeak} yields
\begin{equation*}
f_{\tau(k+1)}^\ast = u^\ast + (f_{\tau(k+1)}^\ast-u^\ast)\in 
s_{\epsilon_{n_{k+1}}/2}^{\rho_{\tee_{n_{k+1}}}(t_{k+1})}\big((1+\nu\sum_{j=1}^{k+1}2^{-j})K\big)
\end{equation*}
hence (I) holds for $m=k+1$.

We now show how to define $x_{\tau(k+1)}$ and then verify that (II) and (III) hold for $m=k+1$. Since $u^\ast\in\yoo_2$ we may write $u^\ast = f_{\tau(k+1)^-}^\ast+ y^\ast +x^\ast$, where $\Vert y^\ast\Vert <\frac{\nu\epsilon_{n_{k+1}}}{4(2+\nu)}$ and $x^\ast \in (\{ x_{\tau(i)}\mid 1\leq i\leq k \}\cup\{ y_{f^\ast}\mid f^\ast\in A\})^\perp$. Since
\[
\Vert x^\ast\Vert \geq \Vert u^\ast - f_{\tau(k+1)^-}^\ast\Vert - \Vert y^\ast\Vert >  \frac{\epsilon_{n_{k+1}}}{4} - \frac{\nu\epsilon_{n_{k+1}}}{4(2+\nu)}>
\frac{(1-\nu)\epsilon_{n_{k+1}}}{4},
\]
an application of \cref{lowerlemma} with $F$, $A$ and $\{ y_{f^\ast}\mid f^\ast\in A \}\subseteq S_X$ as defined above in the current proof yields $y\in S_{F_\perp}$ such that 
\[
\langle x^\ast ,y\rangle > \frac{1}{2+\nu}\frac{(1-\nu)\epsilon_{n_{k+1}}}{4}.
\]
Set $x_{\tau(k+1)}= \vert \langle u^\ast,y\rangle\vert \langle u^\ast,y\rangle^{-1}y$, so that $x_{\tau(k+1)}\in S_{F_\perp}\subseteq S_X$ and
\begin{align*}
\langle u^\ast, x_{\tau(k+1)}\rangle = \vert \langle u^\ast, y\rangle\vert = \vert \langle u^\ast- f_{\tau(k+1)^-}^\ast, y\rangle\vert = \vert \langle x^\ast +y^\ast, y\rangle\vert 
&> \frac{(1-\nu)\epsilon_{n_{k+1}}}{4(2+\nu)} - \frac{\nu\epsilon_{n_{k+1}}}{4(2+\nu)} \\&=\frac{(1-2\nu)\epsilon_{n_{k+1}}}{4(2+\nu)}.
\end{align*}

We now show that (II) holds for $m=k+1$. By the induction hypothesis, we need to prove the case where at least one of $i$ and $j$ is equal to $k+1$. Since $x_{\tau(k+1)}\in S_{F_\perp}\subseteq \bigcap_{j=1}^k\ker(f_{\tau(j)}^\ast)$ we have $\langle f_{\tau(j)}^\ast,x_{\tau(k+1)}\rangle=0$ for $1\leq j\leq k$. Moreover it is clear from the definition of $f_{\tau(k+1)}^\ast$ and the fact that $x_{\tau(k+1)}\in \bigcap_{j=1}^k\ker(f_{\tau(j)}^\ast)$ that 
\[
\langle f_{\tau(k+1)}^\ast,x_{\tau(k+1)}\rangle= \langle u^\ast,x_{\tau(k+1)}\rangle-0 = \langle u^\ast,x_{\tau(k+1)}\rangle >\frac{(1-2\nu)\epsilon_{n_{k+1}}}{4(2+\nu)}.\]
Since (II) holds for $1\leq m\leq k$, if $l,i\in[1,k]$ then
\begin{equation}\label{rumpole}
\langle f_{\tau(l)}^\ast - f_{\tau(l)^-}^\ast,x_{\tau(i)}\rangle=\begin{cases} \langle f_{\tau(i)}^\ast ,x_{\tau(i)}\rangle,&\mbox{if }i=l\\0,&\mbox{if }i\neq l \end{cases}.
\end{equation}
It follows that if $1\leq i\leq k$ then
\begin{align*}
\langle f_{\tau(k+1)}^\ast, x_{\tau(i)}\rangle
&= \langle u^\ast,x_{\tau(i)}\rangle - \frac{\langle u^\ast-f_{\tau(k+1)^-}^\ast,x_{\tau(i)}\rangle}{\langle f_{\tau(i)}^\ast,x_{\tau(i)}\rangle}\langle f_{\tau(i)}^\ast ,x_{\tau(i)}\rangle \notag\\&=\langle f_{\tau(k+1)^-}^\ast,x_{\tau(i)}\rangle\\ &=\begin{cases} \langle f_{\tau(i)}^\ast,x_{\tau(i)}\rangle > \frac{(1-2\nu)\epsilon_{n_i}}{4(2+\nu)}&\mbox{if }\tau(i)\preceq_\mathfrak{T}\tau(k+1) \\ 0&\mbox{if }\tau(i)\npreceq_\mathfrak{T}\tau(k+1)\end{cases},
\end{align*}
hence (II) holds for $m=k+1$.

Finally, we show that (III) holds for $m=k+1$. Let $a$ be a scalar and, to avoid triviality, let $x\in \spn\{ x_{\tau(i)}\mid 1\leq i\leq k\} $ be nonzero. Let $g_x\in G$ be such that $\Vert \Vert x\Vert^{-1}x-g_x\Vert \leq \delta_{k+1}$. Since $\langle h_{g_x}^\ast, g_x\rangle = 1$ and $x_{\tau(k+1)}\in\ker(h_{g_x}^\ast)$ we have
\begin{align*}
\Vert x+a x_{\tau(k+1)}\Vert &\geq \vert \langle h_{g_x}^\ast, x+a x_{\tau(k+1)}\rangle\vert \\&= \vert \langle h_{g_x}^\ast, \Vert x\Vert^{-1}x\rangle\vert \Vert x\Vert\\& \geq \big(\vert \langle h_{g_x}^\ast, g_x\rangle\vert - \vert \langle h_{g_x}^\ast, \Vert x\Vert^{-1}x-g_x\rangle\vert\big) \Vert x\Vert \\&\geq (1-\delta_{k+1})\Vert x\Vert ,
\end{align*}
which completes the proof of part (i) of the theorem.

We now prove (ii). To this end suppose $X^\ast$ is norm separable and let $(z_m^\ast)_{m=1}^\infty$ a norm dense sequence in $X^\ast$. To see that $(x_{\tau(m)})_{m=1}^\infty$ may be chosen to be a shrinking basis, we modify the proof of (i) by extending the list of conditions (I)-(III) to include the following fourth condition:
\begin{itemize}
\item[(IV)] $x_{\tau(m)}\in \bigcap_{j=1}^{m-1}\ker(z_j^\ast)$.
\end{itemize}
In the inductive construction involving the verification of properties (I)-(III), we amend the argument to ensure that (IV) holds for all $m\in[1,\omega)$ as follows. For the base step we require no change in the argument since $\bigcap_{j=1}^0\ker(z_j^\ast)=X$. For the inductive step we change the definition of $F$ so that
\[
F= \spn\big( \{ f_{\tau(i)}^\ast\mid 1\leq i\leq k \} \cup \{ h_g^\ast \mid g\in G\}\cup\{ z_j^\ast\mid 1\leq j\leq k\} \big)\,.
\]
Since $x_{\tau(k+1)}$ is defined so that $x_{\tau(k+1)}\in S_{F_\perp}$, the induction yields that (IV) holds for all $m\in[1,\omega)$. With this modification it is now easy to see that $(x_{\tau(m)})_{m=1}^\infty$ is shrinking. Indeed, let $f^\ast \in [(x_{\tau(m)})_{m=1}^\infty]^\ast$ and fix $\epsilon>0$. Let $\tilde{f}^\ast\in X^\ast$ be an extension of $f^\ast$ to $X$ and let $N<\omega$ be such that $\Vert \tilde{f}^\ast-z_N^\ast\Vert<\epsilon$. By (IV) we have $\langle z_N^\ast,x_{\tau(m)}\rangle=0$ for all $m>N$, hence $m>N$ implies
\[
\Vert f^\ast|_{[x_{\tau(m)}]_{m>N}}\Vert=\Vert \tilde{f}^\ast|_{[x_{\tau(m)}]_{m>N}}\Vert\leq \Vert (\tilde{f}^\ast-z_N^\ast)|_{[x_{\tau(m)}]_{m>N}}\Vert+\Vert z_N^\ast|_{[x_{\tau(m)}]_{m>N}}\Vert<\epsilon\,.
\]
In particular, $\lim_{M\rightarrow\infty}\Vert f^\ast|_{[x_{\tau(m)}]_{m>M}}\Vert =0$. As $f^\ast \in [(x_{\tau(m)})_{m=1}^\infty]^\ast$ was chosen arbitrarily, $(x_{\tau(m)})_{m=1}^\infty$ is shrinking by Proposition 3.2.8 of \textcite{Albiac2016}.

We now show how to modify the proof of (i) further so that $(Qx_{\tau(m)})_{m=1}^\infty$ is a shrinking basis for $X/Z$ with basis constant not exceeding $1+\delta$. In a similar spirit to the proof of Proposition~3.5 of \textcite{Lancien1996}, the main idea is to modify the proof of (i) to incorporate the arguments from Johnson and Rosenthal's proof of Theorem~III.1 of \textcite{Johnson1972}. To this end let $\tripnorm \cdot\tripnorm$ be an equivalent norm on $X$ such that properties (i) and (ii) of \cref{kkthm} hold with $c=(1+\delta)^{1/2}$ and let $(z_p)_{p=1}^\infty$ be a norm dense sequence in $S_{(X,\tripnorm\cdot\tripnorm)}$. Fix a sequence $(\delta_m')_{m=1}^\infty\subseteq (0,1)\subseteq \real$ such that $\sum_{m=1}^\infty\delta_m'<\infty$ and $\prod_{m=1}^\infty (1-\delta_m')\geq (1+\delta)^{-1/2}$. At the $k$th stage of the inductive construction, after having defined the new $f_{\tau(k)}^\ast$, we set $v_k:= \tripnorm f_{\tau(k)}^\ast-f_{\tau(k)^-}^\ast\tripnorm^{-1}(f_{\tau(k)}^\ast-f_{\tau(k)^-}^\ast)$. We modify the inductive construction in proof of (i) to include the construction of a strictly increasing sequence $(p_m)_{m=0}^\infty\subseteq \omega$ such that the following additional conditions hold for $m\in[1,\omega)$:
\begin{itemize}
\item[(V)] For each $v^\ast \in [(v_j)_{j=1}^m]^\ast$ with $\tripnorm v^\ast\tripnorm =1$ there is a natural number $p\leq p_m$ such that $\vert \langle v,z_p\rangle - \langle v^\ast,v\rangle\vert \leq \tripnorm v\tripnorm\delta_m'/3$ for all $v\in [(v_j)_{j=1}^m]$.
\item[(VI)] $\vert \langle v_m, z_p\rangle\vert <\delta_m'/3$ for all $z_p$ with $p\leq p_{m-1}$. 
\end{itemize}
At the base step of the induction we set $p_0=0$ and use Helly's theorem (or Goldstine's theorem), the density of $(z_p)_{p=1}^\infty$ in $S_{(X,\tripnorm\cdot\tripnorm)}$ and the total boundedness of $S_{([\{v_1\}],\tripnorm\cdot\tripnorm)}$ and $S_{([\{v_1\}]^\ast,\tripnorm\cdot\tripnorm)}$ to obtain also $p_1>p_0$ large enough that (V) holds for $m=1$ (we leave the straightforward details to the reader). Since $p_0$ is defined to be $0$, which is not in the index set of the sequence $(z_p)_{p=1}^\infty$, (VI) is true for $m=1$.

At the inductive step of the modified construction we assume that for some $k\in [1,\omega)$ the properties (I)-(VI) hold for $m=1,\ldots k$. We again use Helly's theorem to obtain $p_{k+1}>p_k$ so that (V) holds for $m=k+1$. To obtain that (VI) is true for $m=k+1$ we modify the argument in the proof of (i) as follows. For $1\leq i\leq k$ let $c_i\in\real$ be such that
\[
kc_i\frac{4(2+\nu)}{(1-2\nu)\epsilon_{n_i}} 2( 1+\nu)\sup\{ \Vert y^\ast\Vert\mid y^\ast\in K \}= \frac{1}{2}\frac{\delta_{k+1}'}{3}\frac{(1-2\nu)\epsilon_{n_{k+1}}}{4(2+\nu)}\frac{1}{(1+\delta)^{1/2}}
\] and define
\begin{align*}
\yoo_3&=\bigcap_{i=1}^k\Big\{ x^\ast\in X^\ast\,\,\Big\vert\,\, \vert \langle x^\ast - f_{\tau(k+1)^-}^\ast,x_{\tau(i)}\rangle\vert<c_i \Big\};\mbox{ and,}\\
\yoo_4&=\bigcap_{p=1}^{p_k}\Big\{ x^\ast\in X^\ast\,\,\Big\vert\,\, \vert \langle x^\ast - f_{\tau(k+1)^-}^\ast,z_p\rangle\vert< \frac{1}{2}\frac{\delta_{k+1}'}{3}\frac{(1-2\nu)\epsilon_{n_{k+1}}}{4(2+\nu)}\Big\}.
\end{align*}
We modify the definition of $\yoo$ in the proof of (i) so that $\yoo:= \yoo_1\cap \yoo_2\cap\yoo_3\cap\yoo_4$ and, as in the proof of (i), choose $u^\ast\in\yoo\cap s_{\epsilon_{n_{k+1}}/2}^{\rho_{\tee_{n_{k+1}}}(t_{k+1})}\big((1+\nu\sum_{j=1}^k2^{-j})K\big)$ such that $\Vert f_{\tau(k+1)^-}^\ast - u^\ast\Vert>\epsilon_{n_{k+1}}/4$. Now we define $f_{\tau(k+1)}^\ast$ and $x_{\tau(k+1)}$ as before. It is clear that properties (I)-(III) hold for them (by the same proof as before). Since (II) holds for $1\leq m\leq k$ and since $u^\ast\in\yoo_3$, it follows from the definition of $f_{\tau(k+1)}^\ast$ that 
\[
\Vert f_{\tau(k+1)}^\ast -u^\ast\Vert \leq\frac{1}{2}\frac{\delta_{k+1}'}{3}\frac{(1-2\nu)\epsilon_{n_{k+1}}}{4(2+\nu)}\frac{1}{(1+\delta)^{1/2}}.
\]
Hence, since $u^\ast \in\yoo_4$ we deduce that for all $p\leq p_k$ we have
\begin{align}\label{justcort}
\vert \langle f_{\tau(k+1)}^\ast & -f_{\tau(k+1)^-}^\ast ,z_p\rangle\vert \notag \\&\leq \Vert f_{\tau(k+1)}^\ast -u^\ast\Vert \Vert  z_p\Vert+\vert \langle u^\ast -f_{\tau(k+1)^-}^\ast ,z_p\rangle\vert \notag \\&\leq \frac{1}{2}\frac{\delta_{k+1}'}{3}\frac{(1-2\nu)\epsilon_{n_{k+1}}}{4(2+\nu)}\frac{1}{(1+\delta)^{1/2}}(1+\delta)^{1/2} + \frac{1}{2}\frac{\delta_{k+1}'}{3}\frac{(1-2\nu)\epsilon_{n_{k+1}}}{4(2+\nu)}\notag \\& = \frac{\delta_{k+1}'}{3}\frac{(1-2\nu)\epsilon_{n_{k+1}}}{4(2+\nu)}.
\end{align}
Since \[ \tripnorm f_{\tau(k+1)}^\ast -f_{\tau(k+1)^-}^\ast\tripnorm\geq\Vert f_{\tau(k+1)}^\ast -f_{\tau(k+1)^-}^\ast \Vert \geq \langle f_{\tau(k+1)}^\ast -f_{\tau(k+1)^-}^\ast, x_{\tau(k+1)}\rangle> \frac{(1-2\nu)\epsilon_{n_{k+1}}}{4(2+\nu)},\] it follows from \cref{justcort} that for \[ v_{k+1}:= \tripnorm f_{\tau(k+1)}^\ast-f_{\tau(k+1)^-}^\ast\tripnorm^{-1}(f_{\tau(k+1)}^\ast-f_{\tau(k+1)^-}^\ast)\] we have $\vert \langle v_{k+1},z_p\rangle\vert \leq \delta_{k+1}'/3$ for all $p\leq p_k$. Thus (I)-(VI) hold for all $m\in[1,\omega)$ with these modifications to the proof of (i).

Still following the argument in \textcite{Johnson1972}, our next step is to show that $(v_m)_{m=1}^\infty$ is a basic sequence whose basis constant with respect to $\tripnorm\cdot\tripnorm$ is no larger than $(1+\delta)^{1/2}$. To this end fix $m\in[1,\omega)$ and let $v\in[(v_q)_{q=1}^m]$ be such that $\tripnorm v\tripnorm =1$. Choose $v^\ast\in[v_j]_{1\leq j\leq m}^\ast$ such that $\langle v^\ast ,v\rangle=1=\tripnorm v^\ast\tripnorm$ and choose $p\leq p_m$ so that (V) holds for $v^\ast$. Then $\vert\langle v,z_p\rangle\vert\geq 1-\delta_m'/3$, hence for any scalar $a$ we have \begin{align*}
\tripnorm v+a v_{m+1}\tripnorm &\geq \begin{cases}1 & \mbox{if } \vert a\vert >2\\ \vert \langle v ,z_p\rangle + \langle a v_{m+1},z_p\rangle\vert\geq (1-\frac{\delta_m'}{3})-\frac{2\delta_m'}{3}& \mbox{if }\vert a\vert\leq2
\end{cases}\\
&\geq 1-\delta_m'.
\end{align*}
It follows that $\tripnorm \sum_{q=1}^ma_qv_q\tripnorm\leq \frac{1}{1-\delta_m'}\tripnorm \sum_{q=1}^{m+1}a_qv_q\tripnorm $ for any scalars $a_1,\ldots,a_m,a_{m+1}$. Thus for $1\leq l< m<\omega$ and any scalars $a_1,\ldots,a_m$ we have
\begin{equation}\label{steefryle}
\Bigtripnorm \sum_{q=1}^la_qv_q\Bigtripnorm\leq \frac{1}{\displaystyle\prod_{q=l+1}^m(1-\delta'_q)}\Bigtripnorm \sum_{q=1}^ma_qv_q\Bigtripnorm
\leq (1+\delta)^{1/2}\Bigtripnorm \sum_{q=1}^ma_qv_q\Bigtripnorm.
\end{equation}
By Grunblum's criterion, $(v_m)_{m=1}^\infty$ is a basic sequence whose basis constant with respect to $\tripnorm\cdot\tripnorm$ is no larger than $(1+\delta)^{1/2}$.

Let $(v_m^\ast)_{m=1}^\infty$ be the sequence of functionals in $[(v_m)_{m=1}^\infty]^\ast$ biorthogonal to $(v_m)_{m=1}^\infty$ and define $T: X\longrightarrow [(v_m)_{m=1}^\infty]^\ast$ by $\langle Tx,v\rangle=\langle v,x\rangle$ for $x\in X$ and $v\in[(v_m)_{m=1}^\infty]$. That is, $Tx=(\imath_Xx)|_{[(v_m)_{m=1}^\infty]}$ for each $x\in X$. Note that $\ker(T)=\bigcap_{m=1}^\infty\ker(v_m)$. Moreover since $f_{\tau(0)}^\ast = 0$ we have \[ f_t^\ast=\sum_{\emptyset\prec_\mathfrak{T}s\preceq_\mathfrak{T}t}(f_s^\ast-f_{s^-}^\ast)\] for each $t\in\llbracket\mathfrak{T}\rrbracket$, hence
\[
\ker(T)=\bigcap_{m=1}^\infty\ker(v_m) = \bigcap_{t\in\llbracket\mathfrak{T}\rrbracket^\star}\ker(f_t^\ast-f_{t^-}^\ast)=\bigcap_{t\in\llbracket\mathfrak{T}\rrbracket^\star}\ker(f_t^\ast)=\bigcap_{t\in\llbracket\mathfrak{T}\rrbracket^\star}\ker(x_t^\ast)=Z.
\] 
It is a well known fact that the sequence of coordinate functionals associated with a basis is also a basis with the same basis constant\autocite[See e.g.][Fact 6.6]{Fabian2001}, so let us momentarily denote by $P_m$, $m\geq 1$, the basis projections associated to the basic sequence $(v_m^\ast)_{m=1}^\infty$. Notice that by \cref{steefryle} we have $\tripnorm P_m\tripnorm \rightarrow 1$ as $m\rightarrow\infty$. This together with the condition (V) is enough to prove that $T(X) = [(v_m^\ast)_{m=1}^\infty]$ (see pages 83-84 in \textcite{Johnson1972}). Therefore there exists a linear isometry $\overline{T}:(X/Z,\tripnorm\cdot\tripnorm)\longrightarrow( [(v_m^\ast)_{m=1}^\infty],\tripnorm\cdot\tripnorm)$ such that $\overline{T}Qx=Tx$ for every $x\in X$. We thus have
\[
\overline{T}Qx_{\tau(m)} =Tx_{\tau(m)} = \sum_{m'=1}^\infty\langle Tx_{\tau(m)},v_{m'}\rangle v_{m'}^\ast= \sum_{m'=1}^\infty\langle v_{m'},x_{\tau(m)}\rangle v_{m'}^\ast =\frac{\langle f_{\tau(m)}^\ast,x_{\tau(m)}\rangle}{\tripnorm f_{\tau(m)}^\ast- f_{\tau(m)^-}^\ast \tripnorm}v_m^\ast.
\]
For each $m\in[1,\omega)$ let $a_m= \langle f_{\tau(m)}^\ast,x_{\tau(m)}\rangle\tripnorm f_{\tau(m)}^\ast- f_{\tau(m)^-}^\ast \tripnorm^{-1}$, so that $\overline{T}Qx_{\tau(m)}=a_mv_m^\ast$ for each such $m$. As $\overline{T}$ is an isometry with respect to $\tripnorm\cdot\tripnorm$, then, with respect to $\tripnorm\cdot\tripnorm$, $(Qx_{\tau(m)})_{m=1}^\infty$ is a basis for $X/Z$ isometrically equivalent to $(a_mv_m^\ast)$, whose basis constant coincides with the basis constant of $(v_m^\ast)_{m=1}^\infty$, which coincides with the basis constant of $(v_m)_{m=1}^\infty$, which is no larger than $(1+\delta)^{1/2}$ (as shown above). It follows that for $1\leq l\leq m<\omega$ and any scalars $a_1,\ldots,a_m$ we have
\begin{align*}
\Big\Vert \sum_{q=1}^la_qQx_{\tau(q)}\Big\Vert&\leq(1+\delta)^{1/2}\Bigtripnorm \sum_{q=1}^la_qQx_{\tau(q)}\Bigtripnorm\\&\leq (1+\delta)\Bigtripnorm \sum_{q=1}^ma_qQx_{\tau(q)}\Bigtripnorm \\&\leq (1+\delta)\Big\Vert \sum_{q=1}^ma_qQx_{\tau(q)}\Big\Vert.
\end{align*}
In particular, $(Qx_{\tau(m)})_{m=1}^\infty$ is a basis for $X/Z$ whose basis constant with respect to $\Vert\cdot\Vert$ is no larger than $1+\delta$. It remains then to show that $(Qx_{\tau(m)})_{m=1}^\infty$ is shrinking. As $(Qx_{\tau(m)})_{m=1}^\infty$ is equivalent to $(a_mv_m^\ast)_{m=1}^\infty$, which is shrinking if and only if $(v_m^\ast)_{m=1}^\infty$ is shrinking, it follows from the duality between shrinking and boundedly complete bases\autocite[See e.g.][Corollary~6.1]{Singer1970} that, to complete the proof, it suffices to show that $(v_m)_{m=1}^\infty$ is boundedly complete. To this end recall the following definition from \textcite{Johnson1972}:
\begin{definition}
Let $X$ be a Banach space. A sequence $(y_m^\ast)_{m=1}^\infty\subseteq X^\ast$ is said to be \emph{weak$^\ast$-basic} if there is a sequence $(y_m)_{m=1}^\infty\subseteq X$ so that $(y_m,y_m^\ast)_{m=1}^\infty$ is biorthogonal and for each $y^\ast\in\overline{[y_m^\ast\mid m\in\nat]}^{w^\ast}$ we have $\sum_{q=1}^m\langle y^\ast, y_q\rangle y_q^\ast\stackrel{{weak}^\ast}{\longrightarrow}y^\ast$ as $m\rightarrow \infty$.
\end{definition}
The following result is Proposition~II.1 of \textcite{Johnson1972}.
\begin{proposition}\label{jrprop}
Let $X$ be a Banach space and $(y_m^\ast)_{m=1}^\infty\subseteq X^\ast$. Let $Q$ denote the quotient map $X\longrightarrow X/\bigcap_{m=1}^\infty\ker(y_m^\ast)$. Then
\begin{itemize}
\item[(a)] $(y_m^\ast)_{m=1}^\infty$ is weak$^\ast$-basic if and only if $X/\bigcap_{m=1}^\infty\ker(y_m^\ast)$ has a basis $(e_m)_{m=1}^\infty$ with associated biorthogonal functionals $(e_m^\ast)_{m=1}^\infty$ such that $Q^\ast e_m^\ast=y_m^\ast$ for all $m\in\nat$. It follows that if $(y_m^\ast)_{m=1}^\infty$ is weak$^\ast$-basic, then $(y_m^\ast)_{m=1}^\infty$ is basic.
\item[(b)] The following are equivalent:
\begin{itemize}
\item[(i)] $(y_m^\ast)_{m=1}^\infty$ is a boundedly complete weak$^\ast$-basic sequence;
\item[(ii)] $(y_m^\ast)_{m=1}^\infty$ is weak$^\ast$-basic and $[(y_m^\ast)_{m=1}^\infty]=\overline{[y_m^\ast\mid m\in\nat]}^{w^\ast}$.
\end{itemize}
\end{itemize}
\end{proposition}
We shall first apply (a) of \cref{jrprop} to show that $(v_m)_{m=1}^\infty$ is weak$^\ast$-basic, then apply (b) of \cref{jrprop} to deduce that $(v_m)_{m=1}^\infty$ is boundedly complete, as desired. For $m\in[1,\omega)$ let \[
e_m:=\frac{\tripnorm f_{\tau(m)^\ast}-f_{\tau(m)^-}^\ast\tripnorm}{\langle f_{\tau(m)}^\ast,x_{\tau(m)}\rangle}Qx_{\tau(m)},
\]
so that $\overline{T}e_m=v_m^\ast$. For $m\in[1,\omega)$ let 
$ 
v_m^{\ast\ast}:= (\imath_{[(v_q)_{q=1}^\infty]}v_m)|_{[(v_q^\ast)_{q=1}^\infty]} 
$
and $e_m^\ast:= \overline{T}^\ast v_m^{\ast\ast}$, so that $(v_m^{\ast\ast})_{m=1}^\infty$ and $(e_m^\ast)_{m=1}^\infty$ are the sequences of biothogonal functionals associated to the basic sequences $(v_m^\ast)_{m=1}^\infty$ and $(e_m)_{m=1}^\infty$, respectively. For $1\leq m<\omega$ and $x\in X$ we have
\begin{equation}\label{aldishopping}
\langle Q^\ast e_m^\ast,x\rangle = \langle e_m^\ast,Qx\rangle = \langle v_m^{\ast\ast},\overline{T}Qx\rangle = \langle Tx,v_m\rangle=\langle v_m,x\rangle,
\end{equation}
hence $Q^\ast e_m^\ast=v_m$ for each $m\in[1,\omega)$. By \cref{jrprop}(a), $(v_m)_{m=1}^\infty$ is weak$^\ast$-basic. By \cref{jrprop}(b), to complete the proof of \cref{oopsblah} it now suffices to show that $[(v_m)_{m=1}^\infty]=\overline{[v_m\mid m\in\nat]}^{w^\ast}$.

For each $m\in[1,\omega)$ let $y_m=\langle v_m,x_{\tau(m)}\rangle^{-1}x_{\tau(m)}$ so that, by \cref{biorthogger} and the definition of $v_m$, the system $(y_m,v_m)_{m=1}^\infty \subseteq X\times X^\ast$ is biorthogonal. Following now the proof of Theorem~III.2 in \textcite{Johnson1972}, for each $M\in[1,\omega)$ the operator 
\[
S_M\colon \overline{[v_m\mid m\in\nat]}^{w^\ast}\longrightarrow \overline{[v_m\mid m\in\nat]}^{w^\ast}
\]
given by setting $S_My^\ast = \sum_{m=1}^M\langle y^\ast,y_m\rangle v_m$ for each $y^\ast\in\overline{[v_m\mid m\in\nat]}^{w^\ast}$ satisfies $\tripnorm S_M\tripnorm\leq \prod_{m=M}^\infty \frac{1}{1-\delta_m'}$. Suppose $y^\ast\in\overline{[v_m\mid m\in\nat]}^{w^\ast}$. Then $(S_My^\ast)_{M=1}^\infty$ converges weak$^\ast$ to $y^\ast$ since $(v_m)_{m=1}^\infty$ is weak$^\ast$-basic, hence $\liminf_M\tripnorm S_My^\ast\tripnorm \geq \tripnorm y^\ast\tripnorm$. On the other hand, since $\tripnorm S_M\tripnorm\rightarrow 1$ we have $\limsup_M\tripnorm S_My^\ast\tripnorm \leq \tripnorm y^\ast\tripnorm$. It follows that $\tripnorm S_My^\ast\tripnorm \rightarrow  \tripnorm y^\ast\tripnorm$ as $M\rightarrow \infty$, hence $\tripnorm S_My^\ast -y^\ast \tripnorm \rightarrow 0$ since $\tripnorm \cdot\tripnorm$ satisfies property (ii) of \cref{kkthm}. As $S_My^\ast\in [(v_m)_{m=1}^\infty]$ for all $M$, we conclude that $y^\ast \in [(v_m)_{m=1}^\infty]$, which completes the proof of \cref{oopsblah}.
\end{proof}

The following corollary of \cref{oopsblah} may be useful in situation where one considers the $\epsilon$-Szlenk index for only a single $\epsilon >0$ (rather than a for a sequence $(\epsilon_n)_{n<\omega}$), such as the work in the current paper on universal operators.

\begin{corollary}\label{birthdaybun}
Let $X$ be a Banach space, $K\subseteq X^\ast$ an absolutely convex, weak$^\ast$-compact set, $\epsilon>0$, $\xi>0$ a countable ordinal, and $(\tee,\preceq)$ a countable, well-founded tree such that $\rho(\tee)\leq\xi+1$. If $s_{\epsilon}^\xi(K)\neq \emptyset$ then there exist families $(x_t^\ast)_{t\in\tee}\subseteq K$ and $(x_t)_{t\in\tee}\subseteq S_X$ such that 
\begin{equation}\label{piorpogger}
\langle x_t^\ast ,x_s\rangle=\begin{cases}
\langle x_s^\ast,x_s\rangle>\frac{\epsilon}{17}&\mbox{if }s\preceq t\\ 0 &\mbox{if }s\npreceq t
\end{cases}, \qquad s,t\in \tee.
\end{equation}
\end{corollary}

\begin{proof}
Suppose $s_{\epsilon}^\xi(K)\neq \emptyset$ so that, by \cref{obstaclecentre}, $s_{\epsilon/2}^{\xi+1}(K)\supseteq s_{\epsilon/2}^{\xi2}(K)\neq \emptyset$. Let $t_0$ be a set such that $t_0\notin\tee$ and let $(\tee_0,\preceq_0)$ be the tree obtained by setting $\tee_0=\tee\cup\{t_0\}$ and extending $\preceq$ to $\tee_0$ by making $t_0$ the unique minimal element of $\tee_0$. Let $\xi_0= \xi+1$, so that $\rho(\tee_0)\leq \xi_0+1$ and $s_{\epsilon/2}^{\xi_0}(K)\neq\emptyset$. The conclusion of the corollary follows from an application of \cref{oopsblah}(i) with $\theta=1/2$, $\epsilon_n=\epsilon/2$ for all $n<\omega$, $\xi_n=0$ for $0<n<\omega$, and $(\tee_n,\preceq_n)$ a tree consisting of a single node for $0<n<\omega$ (since $(\llbracket \mathfrak{T}\rrbracket^\star,\preceq_{\mathfrak{T}}) = (\{ 0\} \times \tee,\preceq_{\mathfrak{T}})$ is, in this case, naturally order isomorphic to $\tee$).
\end{proof}

\section{Structure of Banach spaces of large Szlenk index}\label{basissection}
In this section we provide the proofs of \cref{firstbasistheorem} and \cref{drainyourflagon}. In doing so we continue with the notation introduced in \cref{mainlemmasection}. First we prove \cref{firstbasistheorem}.

\begin{proof}[Proof of \cref{firstbasistheorem}]
Fix $\theta\in (0,\sqrt{65}-8)$ and fix $\{\epsilon_n\mid n<\omega\}$, a countable dense subset of $(0,\infty)\subseteq \real$. We apply \cref{oopsblah} with $K=B_{X^\ast}$, $\xi_n=Sz(X,\epsilon_n)-1$ for each $n<\omega$, and $\mathfrak{T}=((\tee_n,\preceq_n))_{n<\omega}$ a family of blossomed trees with $\rho(\tee_n)=Sz(X,\epsilon_n)$ for each $n<\omega$, to obtain families $(x_t)_{t\in\llbracket\mathfrak{T}\rrbracket^\star}\subseteq S_X$ and $(x_t^\ast)_{t\in \llbracket\mathfrak{T}\rrbracket^\star}\subseteq B_{X^\ast}$ such that
\begin{equation}\label{myorthogger}
\langle x_t^\ast ,x_s\rangle=\begin{cases}
\langle x_s^\ast,x_s\rangle>\frac{\epsilon_n}{8+\theta}&\mbox{if }s\preceq_\mathfrak{T} t\in \{n\}\times \tee_n^\star\\ 0 &\mbox{if }s\npreceq_\mathfrak{T} t
\end{cases}, \qquad s,t\in \llbracket\mathfrak{T}\rrbracket^\star,\,n<\omega.
\end{equation}
 
Let $Y_0=[x_t]_{t\in \llbracket\mathfrak{T}\rrbracket^\star}$. It follows from \cref{myorthogger} that for every $n<\omega$ the set $\{ x_{(n,t)}^\ast\mid t\in\tee_n^\star\}$  is $\frac{\epsilon_n}{8+\theta}$-separated and that for every $t\in\tee_n^\star$ the sequence $(x_{(n,u)}^\ast)_{u\in\tee_n[t+]}$ converges weak$^\ast$ in $Y_0^\ast$ to $x_{n,t}^\ast$. Now an easy inductive argument (similar to that used to prove \cref{nearlycrawling}) shows that\[ \forall \,n<\omega\quad \forall\, (n,t)\in \{ n\} \times\tee_n^\star\qquad  x_{(n,t)}^\ast\vert_{Y_0}\in s_{\epsilon_n/(8+\theta)}^{\rho_{\tee_n^\star}(t)}(B_{Y_0^\ast}).\] Hence, since $(x_{(n,u)}^\ast)_{u\in\tee_n[\emptyset+]}$ converges weak$^\ast$ in $Y_0^\ast$ to $0$, we obtain that \[0\in s_{\epsilon_n/(8+\theta)}^{\rho(\tee_n)-1}(B_{Y_0^\ast})= s_{\epsilon_n/(8+\theta)}^{Sz(X,\epsilon_n)-1}(B_{Y_0^\ast})\] for each $n<\omega$. It follows that
\begin{equation}\label{fauxbounder}
\forall \,n<\omega \quad Sz\Big(Y_0,\frac{\epsilon_n}{8+\theta}\Big)\geq Sz(X,\epsilon_n)\,.
\end{equation}
Let $\tau:\omega\longrightarrow \llbracket\mathfrak{T}\rrbracket$ be a bijection such that $\tau(m)\preceq_\mathfrak{T}\tau(m')$ implies $m\leq m'$. Since $Sz(Y_0)<\omega_1$ and $Y_0$ is separable, the dual $Y_0^\ast$ is norm separable. Thus, by \cref{fauxbounder} and \cref{oopsblah} there exist families $(y_t)_{t\in\llbracket\mathfrak{T}\rrbracket^\star}\subseteq S_{Y_0}$ and $(y_t^\ast)_{t\in \llbracket\mathfrak{T}\rrbracket^\star}\subseteq B_{Y_0^\ast}$ such that
\begin{equation}\label{quyorthogger}
\langle y_t^\ast ,y_s\rangle=\begin{cases}
\langle y_s^\ast,y_s\rangle>\frac{\epsilon_n}{65}&\mbox{if }s\preceq_\mathfrak{T} t\in \{n\}\times \tee_n^\star\\ 0 &\mbox{if }s\npreceq_\mathfrak{T} t
\end{cases}, \qquad s,t\in \llbracket\mathfrak{T}\rrbracket^\star,\,n<\omega.
\end{equation}
and $(y_{\tau(m)})_{m=1}^\infty$ is a shrinking basic sequence with basis constant not exceeding $1+\delta$.
Let $Y=[ (y_{\tau(m)})_{m=1}^\infty]$. It follows from \cref{quyorthogger} that\[ \forall \,n<\omega\quad \forall\, (n,t)\in \{ n\} \times\tee_n^\star\qquad  y_{(n,t)}^\ast\vert_Y\in s_{\epsilon_n/65}^{\rho_{\tee_n^\star}(t)}(B_{Y^\ast})\] and, subsequently, that $0\in s_{\epsilon_n/65}^{\rho(\tee_n)-1}(B_{Y^\ast})= s_{\epsilon_n/65}^{Sz(X,\epsilon_n)-1}(B_{Y^\ast})$ for each $n<\omega$. Thus,
\begin{equation}\label{bauxbounder}
\forall \,n<\omega \quad Sz\Big(Y,\frac{\epsilon_n}{65}\Big)\geq Sz(X,\epsilon_n)\,.
\end{equation}

For each $\epsilon>0$ choose $N(\epsilon)<\omega$ such that $\epsilon_{N(\epsilon)}\in [\frac{65\epsilon}{66},\epsilon]$. From \cref{bauxbounder} we obtain
\[
\forall\,\epsilon>0\quad Sz\Big(Y,\frac{\epsilon}{66}\Big)\geq Sz\Big(Y,\frac{\epsilon_{N(\epsilon)}}{65}\Big)\geq Sz(X,\epsilon_{N(\epsilon)})\geq Sz(X,\epsilon),
\]
which completes the proof of the theorem.
\end{proof}

Two applications of \cref{oopsblah} were used in the proof of \cref{firstbasistheorem} - the first to achieve separable reduction and the second to obtain a \emph{shrinking} basic sequence. Clearly, if $X$ is assumed norm separable then only one application of \cref{oopsblah} is required, in which case the number $65$ in \cref{wuppertare} may be replaced by $8+\theta$ for any $\theta>0$. Moreover, in the general case we may replace $65$ by $16+\theta$ for any $\theta>0$; this is achieved by proving a version of Lemma~3.4 of \textcite{Lancien1996} for families $(\epsilon_n)_{n<\omega}\subseteq (0,\infty)$ and blossomed trees $((\tee_n,\preceq_n))_{n<\omega}$ (as in the proof \cref{oopsblah}), then applying this generalisation of Lemma~3.4 of \textcite{Lancien1996} to achieve separable reduction in the proof of \cref{firstbasistheorem} with $\epsilon/2$ (rather than $\epsilon/(8+\theta)$) replacing $\epsilon$.

We now prove \cref{drainyourflagon}.

\begin{proof}[Proof of \cref{drainyourflagon}]
Fix a countable, dense subset $\{\epsilon_n\mid n<\omega\}$ of $(0,\infty)$. We apply \cref{oopsblah} with $\theta=\frac12$, $K=B_{X^\ast}$, $\xi_n=Sz(X,\epsilon_n)-1$ for each $n<\omega$, and $\mathfrak{T}=((\tee_n,\preceq_n))_{n<\omega}$ a family of blossomed trees with $\rho(\tee_n)=Sz(X,\epsilon_n)$ for each $n<\omega$. Let $\tau:\omega\longrightarrow \llbracket\mathfrak{T}\rrbracket$ be a bijection such that $\tau(m)\preceq_\mathfrak{T}\tau(m')$ implies $m\leq m'$. By \cref{oopsblah} there exist families $(x_t)_{t\in\llbracket\mathfrak{T}\rrbracket^\star}\subseteq S_X$ and $(x_t^\ast)_{t\in \llbracket\mathfrak{T}\rrbracket^\star}\subseteq B_{X^\ast}$ such that
\begin{equation}\label{spyorthogger}
\langle x_t^\ast ,x_s\rangle=\begin{cases}
\langle x_s^\ast,x_s\rangle>\frac{2\epsilon_n}{17}&\mbox{if }s\preceq_\mathfrak{T} t\in \{n\}\times \tee_n^\star\\ 0 &\mbox{if }s\npreceq_\mathfrak{T} t
\end{cases}, \qquad s,t\in \llbracket\mathfrak{T}\rrbracket^\star,\,n<\omega.
\end{equation}
and $(Qx_{\tau(m)})_{m=1}^\infty$ is a shrinking basis for $X/\bigcap_{t\in\llbracket\mathfrak{T}\rrbracket^\star}\ker(x_t^\ast)$ with basis constant not exceeding $1+\delta$, where $Q:X\longrightarrow X/\bigcap_{t\in\llbracket\mathfrak{T}\rrbracket^\star}\ker(x_t^\ast)$ is the quotient map. Let $Z=\bigcap_{t\in\llbracket\mathfrak{T}\rrbracket^\star}\ker(x_t^\ast)$. To complete the proof we will show that \cref{scupperscare} holds.

We may assume that the families $(x_t)_{t\in\llbracket\mathfrak{T}\rrbracket^\star}$ and $(x_t^\ast)_{t\in\llbracket\mathfrak{T}\rrbracket^\star}$ above are those constructed in the proof of \cref{oopsblah}. Let $(f_t^\ast)_{t\in\llbracket\mathfrak{T}\rrbracket^\star}$ and $(v_m)_{m=1}^\infty$ also be as in the proof of \cref{oopsblah}. We have
\begin{equation}\label{couldntbe}
\spn\{ x_t^\ast \mid t\in\llbracket\mathfrak{T}\rrbracket^\star\} = \spn\{ f_t^\ast\mid t\in\llbracket\mathfrak{T}\rrbracket^\star\} = \spn\{ v_m\mid 1\leq m<\omega\}\subseteq Q^\ast \big( (X/Z)^\ast\big),
\end{equation}
where the first equality is immediate from the definitions, the second equality follows from the inductively verified fact that
\[
\forall\,k<\omega \quad \spn\{ f_{\tau(m)}^\ast\mid 1\leq m\leq k\}=\spn\{ v_m\mid 1\leq m\leq k\},
\] and the final inclusion follows from \cref{aldishopping}. Since $\Vert Q\Vert = 1$ and since $Q^\ast$ is an isometric embedding it follows respectively that $Qx_t\in B_{X/Z}$ and that $(Q^\ast)^{-1}(x_t^\ast)$ is a well-defined element of $B_{(X/Z)^\ast}$ for every $t\in \llbracket\mathfrak{T}\rrbracket^\star$. Since for $s,t\in \llbracket\mathfrak{T}\rrbracket^\star $ and $n<\omega$ we have
\begin{equation*}
\langle (Q^\ast)^{-1}(x_t^\ast) ,Qx_s\rangle=\langle x_t^\ast ,x_s\rangle=\begin{cases}
\langle x_s^\ast,x_s\rangle>\frac{2\epsilon_n}{17}&\mbox{if }s\preceq_\mathfrak{T} t\in \{n\}\times \tee_n^\star\\ 0 &\mbox{if }s\npreceq_\mathfrak{T} t
\end{cases}
\end{equation*}
and, since $\spn\{ Qx_t\mid t\in \llbracket\mathfrak{T}\rrbracket^\star\}$ is norm dense in $X/Z$, we have that the family $\{ (Q^\ast)^{-1}(x_{(n,t)}^\ast)\mid t\in\tee_n^\ast \}$ is $\frac{2\epsilon_n}{17}$-separated and that for every $t\in\tee_n$ the sequence $((Q^\ast)^{-1}(x_{(n,u)}^\ast))_{u\in\tee_n[t+]}$ converges weak$^\ast$ in $(X/Z)^\ast$ to $(Q^\ast)^{-1}(x_{(n,t)}^\ast)$. Now an easy inductive argument (similar to that used to prove \cref{nearlycrawling}) yields\[ \forall \,n<\omega\quad \forall\, (n,t)\in \{ n\} \times\tee_n^\star\qquad  (Q^\ast)^{-1}(x_{(n,t)}^\ast)\in s_{2\epsilon_n/17}^{\rho_{\tee_n^\star}(t)}(B_{(X/Z)^\ast})\] and, subsequently, that $0\in s_{2\epsilon_n/17}^{\rho(\tee_n)-1}(B_{(X/Z)^\ast})= s_{2\epsilon_n/17}^{Sz(X,\epsilon_n)-1}(B_{(X/Z)^\ast})$ for each $n<\omega$. Thus,
\begin{equation}\label{sauxbounder}
\forall \,n<\omega \quad Sz\Big(X/Z,\frac{2\epsilon_n}{17}\Big)\geq Sz(X,\epsilon_n)\,.
\end{equation}
For each $\epsilon>0$ choose $N(\epsilon)<\omega$ such that $\epsilon_{N(\epsilon)}\in [\frac{17\epsilon}{18},\epsilon]$. From \cref{sauxbounder} we obtain
\[
\forall\,\epsilon>0\quad Sz\Big(X/Z,\frac{\epsilon}{9}\Big)\geq Sz\Big(X/Z,\frac{2\epsilon_{N(\epsilon)}}{17}\Big)\geq Sz(X,\epsilon_{N(\epsilon)})\geq Sz(X,\epsilon),
\]
which completes the proof of the theorem.
\end{proof}

\section{Universal operators of large Szlenk index}\label{universalsection}
The first and main undertaking of this section is the proof of \cref{countuniv} which provides the classification of the ordinals $\beta$ for which the class $\complement\szlenkop_\beta$ admits a universal element.

\begin{proof}[Proof of \cref{countuniv}]
Fix $\epsilon'>0$ small enough that $s_{\epsilon'}^{\omega^\alpha}(T^\ast (B_{Y^\ast}))\neq\emptyset$ and fix $N<\omega$ large enough that $\rho(\tee) \leq \omega^\alpha 2^N+1$. Set $\epsilon = 2^{-N-1}\epsilon'$, so that $s_\epsilon^{\rho(\tee)-1}(T^\ast (B_{Y^\ast}))\neq \emptyset$ by \cref{dingraniel}. \cref{birthdaybun} yields families $(x_t)_{t\in\tee^\star}\subseteq S_X$ and $(x_t^\ast)_{t\in\tee^\star}\subseteq T^\ast B_{Y^\ast}$ such that 
\begin{equation*}
\langle x_t^\ast ,x_s\rangle=\begin{cases}
\langle x_s^\ast,x_s\rangle>\frac{\epsilon}{17}&\mbox{if }s\preceq t\\ 0 &\mbox{if }s\npreceq t
\end{cases}, \quad s,t\in \tee^\star\,.
\end{equation*}
By \cref{factorkar}, $\Sigma_{\tee^\star}$ factors through $T$.

We now suppose that $\tee$ is blossomed and $\rho(\tee)\geq \omega^\alpha$. Since $\tee$ is infinite and rooted we have $\rho(\tee)\geq2$ and that $\rho(\tee)$ is a successor ordinal. It follows that $\rho(\tee)>\omega^\alpha$, hence by \cref{lbszlenk} we have $Sz(\Sigma_{\tee^\star}) \geq Sz(\Sigma_{\tee^\star},\epsilon)\geq\rho(\tee)>\omega^\alpha, $
so that $\Sigma_{\tee^\star}\in\complement\szlenkop_\alpha$. It follows that $\Sigma_{\tee^\star}$ is universal for $\complement \szlenkop_\alpha$.

Finally, let $\beta$ be an arbitrary ordinal. If $\beta<\omega_1$ then, by the second assertion of \cref{countuniv}, $\Sigma_{\tee_{\omega^\beta}^\star}$ is universal for $\complement\szlenkop_\beta$, where $\tee_{\omega^\beta}$ is as constructed in \cref{blossexist}. Now suppose on the other hand that $\beta\geq\omega_1$; to complete the proof we show that $\complement\szlenkop_\beta$ does not admit a universal element. Suppose by way of contraposition that $\complement\szlenkop_\beta$ \textit{does} admit a universal element, $\Upsilon$ say. Since\autocite[][Theorem~2.6]{Brooker2013} $Sz(C(\omega^{\omega^\beta}+1))= \omega^{\beta+1}$, where $C(\omega^{\omega^\beta}+1)$ denotes the Banach space of continuous scalar-valued functions on the compact ordinal $\omega^{\omega^\beta}+1$, we have that $\Upsilon$ factors through the identity operator of $C(\omega^{\omega^\beta}+1)$. It follows that $Sz(\Upsilon)$ satisfies $Sz(\Upsilon)\leq Sz(C(\omega^{\omega^\beta}+1))= \omega^{\beta+1}$ by the ideal property of $\szlenkop_{\beta+1}$. Moreover the identity operator of $\ell_1$ belongs to $\complement\szlenkop_\beta$ since $\ell_1$ is not an Asplund space, hence $\Upsilon$ factors through $\ell_1$ and, in particular, $\Upsilon$ has separable range. It thus follows by \cref{sepsepseppy} that $Sz(\Upsilon)<\omega^{\omega_1}=\omega_1$, hence $\Upsilon\in \szlenkop_{\omega_1}\subseteq \szlenkop_{\beta}$ - a contradiction. Thus $\complement\szlenkop_\beta$ does not admit a universal element whenever $\beta\geq\omega_1$.
\end{proof}

In \cref{nonspear} we will study whether each of the classes $\asplundop\cap\complement\szlenkop_\beta$, $\beta\geq\omega_1$, admits a universal element of the form $\Sigma_\tee$ for an uncountable tree $\tee$.

\begin{remark}\label{softerbristles}
It is straightforward to observe that we may replace $\Sigma_{\tee^\star}$ by $\Sigma_\tee$ in the statement of \cref{countuniv}. However, the reason for our choice of $\Sigma_{\tee^\star}$ over $\Sigma_\tee$ is that universal operators may be thought of as `minimal' elements of the class for which they are universal, and the operator $\Sigma_{\tee^\star}$ can be thought of as naturally `smaller' than $\Sigma_\tee$ since $\tee^\star$ is a subtree of $\tee$ and $\Sigma_{\tee_\star}$ therefore factors through $\Sigma_\tee$ by \cref{subtreefac}. Moreover, $\tee$ is not order isomorphic to a subtree of $\tee^\star$ since $\tee$ is assumed to be well-founded.
\end{remark}

The following result can be proved directly using standard Szlenk index techniques, however as we shall refer to this result later we provide a quick proof here using \cref{countuniv}.

\begin{corollary}\label{spoindex}
Let $(\tee,\preceq)$ be an infinite blossomed tree. Then $Sz(\Sigma_{\tee^\star})= \rho(\tee)\omega$.
\end{corollary}
\begin{proof}
Let $\alpha<\omega_1$ be the (unique) ordinal satisfying $\omega^\alpha\leq \rho(\tee)<\omega^{\alpha+1}$. Since $Sz(C(\omega^{\omega^\alpha}+1))=\omega^{\alpha+1}>\omega^\alpha$ by Samuel's computation\autocite{Samuel1984} of $Sz(C(K))$ for countable compact Hausdorff $K$, by \cref{countuniv} we have that $\Sigma_{\tee^\star}$ factors through $C(\omega^{\omega^\alpha}+1)$, hence $Sz(\Sigma_{\tee^\star})\leq Sz(C(\omega^{\omega^\alpha}+1))=\omega^{\alpha+1}=\rho(\tee)\omega$ by the ideal property of $\szlenkop_{\alpha+1}$.

As $\tee$ is infinite and rooted we have $\rho(\tee)\geq 2$. Moreover, as noted in \cref{treesubsection}, $\rho(\tee)$ is a successor ordinal. It follows that $\rho(\tee)>\omega^\alpha$, hence $Sz(\Sigma_{\tee^\star})\geq Sz(\Sigma_{\tee^\star},\epsilon)\geq \rho(\tee)>\omega^\alpha$ by \cref{lbszlenk}. As $Sz(\Sigma_{\tee^\star})$ is a power of $\omega$, we deduce that $Sz(\Sigma_{\tee^\star})\geq\omega^{\alpha+1}=\rho(\tee)\omega$, which completes the proof.
\end{proof}

\begin{remark} \cref{Jthm} may be obtained as an immediate consequence of \cref{countuniv} and the fact that $\compactop=\szlenkop_0$ by Proposition~2.3 of \textcite{Brooker2012}. \end{remark}

The following proposition relates some of the factorisation results of the current paper to known relationships between various closed operator ideals.
\begin{proposition}\label{advent}
Let $I$ be a cofinal subset of $\omega_1$ and for each ordinal $\xi\in I$ let $\tee_\xi$ be a blossomed tree with $\rho(\tee_\xi)=\xi+1$. For $T\in\sepop$, i.e. $T$ with separable range, the following are equivalent:
\begin{itemize}
\item[{(i)}] $T$ factors $\Sigma_{\tee_\xi^\star}$ for every $\xi\in I$.
\item[{(ii)}] $T$ factors $\Sigma_\Omega$, where $\Omega$ is the full countably branching tree of \cref{fullcounter}.
\item[{(iii)}] $T$ factors $\Sigma_\tee$ for every countable tree $\tee$ with $ht(\tee)\leq\omega$.
\item[{(iv)}] $T \in \complement \asplundop$.
\item[{(v)}] $T\in\complement \sepop^\ast$.
\item[{(vi)}] $T\in\complement \szlenkop_{\omega_1}$.
\item[{(vii)}] $T\in\complement \bigcup_{\alpha\in\ord}\szlenkop_\alpha$.
\item[{(viii)}] $T\in\complement \bigcup_{\alpha<\omega_1}\szlenkop_\alpha$.
\end{itemize}
\end{proposition}

The proof of \cref{advent} relies on the following result of the author\autocite[][Theorem~5.5]{Brooker2017}.

\begin{theorem}\label{puniuni}
Let $X$ and $Y$ be Banach spaces and $T\in\allop(X,Y)$. Suppose that at least one of $X$ and $Y$ is separable and that $T\notin\sepop^\ast(X,Y)$. Then $\Sigma_\Omega$ factors through $T$.
\end{theorem}

\begin{proof}[Proof of \cref{advent}]
The equivalence of assertions (iv)-(viii) follows immediately from \cref{sepsepseppy}. To complete the proof it thus suffices to show that (v)$\Rightarrow$(ii)$\Rightarrow$(iii)$\Rightarrow$(i)$\Rightarrow$(viii).

To see that (v)$\Rightarrow$(ii), let $X$ and $Y$ be Banach spaces and $T\in\sepop(X,Y)\setminus \sepop^\ast(X,Y)$. Let $\tilde{T}: X\longrightarrow \overline{T(X)}$ be the separable codomain operator given by setting $\tilde{T}x=Tx$ for all $x\in X$. Since $\tilde{T}\notin\sepop^\ast$, by \cref{puniuni} there exist $U\in\allop (\ell_1(\Omega), X)$ and $V\in\allop(\overline{T(X)}, \ell_\infty(\Omega))$ such that $V\tilde{T}U=\Sigma_\Omega$. Since $\ell_\infty(\Omega)$ is injective\autocite[See][p.105]{Lindenstrauss1977} $V$ admits a continuous linear extension $\tilde{V}\in\allop(Y,\ell_\infty(\Omega))$, and for such $\tilde{V}$ we have $\tilde{V}TU=\Sigma_\Omega$. Thus (v)$\Rightarrow$(ii).

The implication (ii)$\Rightarrow$(iii) follows from \cref{subtreefac} and the fact that every countable tree $\tee$ with $ht(\tee)\leq\omega$ is order isomorphic to a subtree of $\Omega$, whilst (iii)$\Rightarrow$(i) is immediate from the fact that blossomed trees are by definition countable and well-founded. 

Finally, the implication (i)$\Rightarrow$(viii) is a consequence of \cref{lbszlenk}.
\end{proof}

\section{When the codomain is separable}\label{kneesbees}

It is evident from the definition of the operator $\Sigma_\tee$ associated to a tree $(\tee,\preceq)$ that the range of $\Sigma_\tee$ is contained in the closed linear span in $\ell_\infty(\tee)$ of the indicator functions $\chi_{\tee[t\preceq]}$, for $t\in\tee$. For example, the range of the universal non-compact operator $\ell_1\hookrightarrow\ell_\infty$ of Johnson (c.f. \cref{Jthm}) is contained in the subspace $c_0$ of $ \ell_\infty$, whilst the range of the Lindenstrauss-Pe{\l}czy{\'n}ski universal non-weakly compact summation operator from $\ell_1$ to $\ell_\infty$ (c.f. \cref{LPthm}) is contained in the subspace $c$ of $\ell_\infty$ consisting of all convergent scalar sequences. In both these cases, the range is contained (up to isometric isomorphism) in a separable $C(K)$ space. In the corresponding papers of Johnson\autocite{Johnson1971a} and Lindenstrauss-Pe{\l}czy{\'n}ski\autocite{Lindenstrauss1968} it is noted that stronger versions of the universal operator theorems presented there hold under restriction to the class of operators having separable codomain. More precisely, it is noted in \textcite{Johnson1971a} that if $T:X\longrightarrow Y$ is noncompact and $Y$ is separable, then $T$ factors the formal identity operator from $\ell_1$ to $c_0$. Moreover, in \textcite{Lindenstrauss1968} it is noted that if $T:X\longrightarrow Y$ is non-weakly compact and $Y$ is separable, then $T$ factors the summation operator from $\ell_1$ to $c$ defined by $(a_n)_{n=1}^\infty\mapsto (\sum_{i=1}^na_i)_{i=1}^n$. In a similar vein, in the current section we show that for every $\alpha<\omega_1$ there exists an operator $\Upsilon_\alpha$ from $\ell_1$ into a separable $C(K)$ space with $Sz(\Upsilon_\alpha)>\omega^\alpha$ and such that $\Upsilon_\alpha$ factors through any $T:X\longrightarrow Y$ with $Y$ separable and $Sz(T)\nleq\omega^\alpha$.

Let $(\tee,\preceq)$ be a rooted and chain-complete tree and let $t_\emptyset$ denote the root of $\tee$. By \cref{coarsecompact} the coarse wedge topology of $\tee$ is compact Hausdorff, thus for such $\tee$ we shall denote by $C(\tee)$ the Banach space (with the supremum norm) of coarse-wedge-continuous scalar-valued functions on $\tee$. We denote by $C_0(\tee)$ the codimension-$1$ subspace $\{ f\in C(\tee)\mid f(t_\emptyset)=0\}$ of $C(\tee)$. Note that if $\tee$ is countable then $C(\tee)$ and $C_0(\tee)$ are norm separable.

The following definition establishes the class of operators from which we shall draw our examples of universal non-$\alpha$-Szlenk operators with separable codomain. We note that although the definition can be adapted to trees of arbtrarily large height, such generality is unnecessary for our purposes. 

\begin{definition}\label{trootedree}
Let $(\tee,\preceq)$ be a rooted, well-founded tree. Define $\sigma_\tee: \ell_1(\tee)\longrightarrow C(\tee)$ by
\begin{equation*}
(\sigma_\tee x)(t)=
\sum_{s\preceq t}x(s), \quad x\in\ell_1(\tee), \,t\in\tee.
\end{equation*}
That is, $\sigma_\tee$ is the unique element of $\allop(\ell_1(\tee),C(\tee))$ that maps each $e_t\in\ell_1(\tee)$ to $\chi_{\tee[t\preceq]}\in C(\tee)$. Similarly, define $\mathring{\sigma}_\tee$ to be the unique element of $\allop( \ell_1(\tee^\star), C_0(\tee))$ that maps each $e_t\in\ell_1(\tee^\star)$ to $\chi_{\tee[t\preceq]}\in C_0(\tee)$.
\end{definition}

Notice that \cref{lbszlenk} holds true with $\mathring{\sigma}_\tee$ in place of $\Sigma_{\tee^\star}$. Indeed, since $C_0(\tee)$ naturally embeds linearly and isometrically into $ \ell_\infty(\tee^\star)$ via the restriction map $R\in\allop(C_0(\tee),\ell_\infty(\tee^\star))$, defined by setting $R(f)=f|_{\tee^\star}$ for each $f\in C_0(\tee)$, and since $\Sigma_{\tee^\star}=R\mathring{\sigma}_\tee$, we have $\mathring{\sigma}_\tee^\ast (B_{C_0(\tee)^\ast})= \Sigma_{\tee^\star}^\ast (B_{\ell_\infty(\tee^\star)^\ast})$. We thus deduce that $Sz(\mathring{\sigma}_\tee)=Sz(\Sigma_{\tee^\star})$ since the Szlenk indices of $\mathring{\sigma}_\tee$ and $\Sigma_{\tee^\star}$ are determined by the same subset of $\ell_1(\tee^\star)^\ast$.

Similarly to the comments in \cref{softerbristles} regarding \cref{countuniv}, we note that although $\mathring{\sigma}_\tee$ may be replaced by $\sigma_\tee$ in the statement of \cref{sepseekay}, we present \cref{sepseekay} as stated since $\mathring{\sigma}_\tee$ may be viewed as being naturally `smaller' than $\sigma_\tee$.

A smallness condition of some kind on $Y$ is necessary for the first assertion of \cref{sepseekay} to hold in general. To see this, by \cref{spoindex} it is enough to observe that for a countably infinite, rooted, well-founded tree $(\tee,\preceq)$, $\mathring{\sigma}_\tee$ does not factor through $\Sigma_{\tee^\star}$. Firstly, the fact that such $\tee$ is infinite and well-founded implies that $\MAX(\tee)$ contains an infinite anti-chain $\{ t_n\mid n<\omega\}$, so that $\mathring{\sigma}_\tee$ is non-compact since the set $\{ \mathring{\sigma}_\tee e_{t_n}\mid n<\omega\}$ is an infinite $\epsilon$-separated subset of $\mathring{\sigma}_\tee (B_{\ell_1(\tee^\star)})$ for any $\epsilon\in(0,1)$. Secondly, the norm separability of $C_0(\tee)$ and the fact that $\ell_\infty(\tee^\star)$ is a Grothendieck space implies that\autocite{Grothendieck1953} $\allop(\ell_\infty(\tee^\star), C_0(\tee))= \wcompactop(\ell_\infty(\tee^\star),C_0(\tee))$. Thirdly, the fact that $\tee$ is well-founded implies that $\Sigma_{\tee^\star}$ is weakly compact by \cref{otherclasses}, hence for any $V\in\allop(\ell_\infty(\tee^\star), C_0(\tee))= \wcompactop(\ell_\infty(\tee^\star),C_0(\tee))$ we have that $V\Sigma_{\tee^\star}$ is compact since $\ell_\infty(\tee^\star)$ is isomorphic to a $C(K)$ space and therefore has the Dunford-Pettis Property\autocite{Grothendieck1953}\autocite[See also][Theorem~5.4.6]{Albiac2016}. Finally, since $V\Sigma_{\tee^\star}$ therefore cannot factor $\mathring{\sigma}_\tee$ for any $V\in\allop(\ell_\infty(\tee^\star), C_0(\tee))$, we conclude that $\Sigma_{\tee^\star}$ does not factor $\mathring{\sigma}_\tee$.

To prove \cref{sepseekay} we first establish the following continuous analogue of \cref{factorkar}.

\begin{proposition}\label{kaktorfar}
Let $K$ be a compact Hausdorff space, $I$ an index set and $\{ K_i\}_{i\in I}$ a family of clopen subsets of $K$. For Banach spaces $X$ and $Y$ and $T\in\allop(X,Y)$ the following are equivalent:
\begin{itemize}
\item[(i)] $T$ factors the unique element of $\allop(\ell_1(I), C(K))$ satisfying $e_i\mapsto \chi_{K_i}$, $i\in I$.
\item[(ii)] There exists $(\delta_i)_{i\in I}\subseteq \real$ with $\inf_{i\in I}\delta_i>0$, a family $(x_i)_{i\in I}\subseteq X$ with $\sup_{i\in I}\Vert x_i\Vert <\infty$ and a weak$^\ast$-continuous map $\Xi: K\longrightarrow Y^\ast$ such that
\[
\forall\, i\in I\quad \forall\, k\in K \quad \langle \Xi (k), Tx_i\rangle = \begin{cases}
\delta_i, & k\in K_i\\
0, &k\notin K_i
\end{cases}.
\]
\end{itemize}
\end{proposition}

\begin{proof}
For each $k\in K$ let $g_k^\ast\in C(K)^\ast$ denote the evaluation functional of $C(K)$ at $k$; that is, $g_k^\ast(f)=f(k)$ for each $f\in C(K)$.

Suppose (i) holds. Let $S$ denote the unique element of $\allop(\ell_1(I), C(K))$ satisfying $e_i\mapsto \chi_{K_i}$, $i\in I$, and let $U\in\allop(\ell_1(I),X)$ and $V\in\allop(Y,C(K))$ be such that $S=VTU$. The map $k\mapsto g_k^\ast$ is a homeomorphic embedding of $K$ into $C(K)^\ast$ with respect to the weak$^\ast$-topology of $C(K)^\ast$, hence the map $\Xi: K\longrightarrow Y^\ast$ defined by setting $\Xi(k) = V^\ast g_k^\ast$ for each $k\in K$ is weak$^\ast$-continuous. For each $i\in I$ set $x_i= Ue_i$, so that $\sup_{i\in I}\Vert x_i\Vert\leq \Vert U\Vert <\infty$. Then for $i\in I$ and $k\in K$ we have
\[
\langle \Xi (k), Tx_i\rangle = \langle V^\ast g_k^\ast, TUe_i\rangle= \langle g_k^\ast, VTUe_i\rangle = \langle g_k^\ast, Se_i\rangle =\langle g_k^\ast, \chi_{K_i}\rangle =\begin{cases}
1, & k\in K_i\\
0, &k\notin K_i
\end{cases}.
\]
By taking $\delta_i=1$ for each $i\in I$ we see that (ii) holds, as desired.

Now suppose (ii) holds. Let $U\in\allop(\ell_1(I), X)$ be defined by setting $Ue_i=\delta_i^{-1}x_i$ for each $i\in I$, noting that $U$ is well-defined with $\Vert U\Vert \leq (\inf_{i\in I}\delta_i)^{-1} \sup_{i\in I}\Vert x_i\Vert$. Let $V$ be the element of $\allop(Y, C(K))$ satisfying $(Vy)(k) = \langle \Xi(k), y\rangle$ for $y\in Y$ and $k\in K$, noting that $V$ is well-defined with $\Vert V\Vert =\sup_{k\in K}\Vert \Xi(k)\Vert <\infty$. For $i\in I$ and $k\in K$ we have
\[
(VTUe_i)(k) = \langle g_k^\ast, VTUe_i\rangle = \delta_i^{-1}\langle V^\ast g_k^\ast, Tx_i\rangle = \delta_i^{-1}\langle \Xi (k), Tx_i\rangle = \begin{cases}
1, & k\in K_i\\
0, &k\notin K_i
\end{cases},
\]
hence $VTUe_i=Se_i$ for every $i\in I$, hence $VTU=S$.
\end{proof}

The following lemma is another key ingredient required for the proof of \cref{sepseekay}.

\begin{lemma}\label{seasonover}
Let $X$ and $Y$ be Banach spaces such that $Y$ is separable. Let $T\in\allop(X,Y)$, $\delta>0$, $(\arr,\preceq')$ a blossomed tree and $(x_r)_{r\in\arr}\subseteq S_X$ and $(x_r^\ast)_{r\in\arr}\subseteq T^\ast (B_{Y^\ast})$ families such that
\begin{equation}\label{quiorquogger}
\langle x_s^\ast ,x_r\rangle=\begin{cases}
\langle x_r^\ast,x_r\rangle>\delta&\mbox{if }r\preceq' s\\ 0 &\mbox{if }r\npreceq' s
\end{cases}, \qquad r,s\in\arr.
\end{equation}
Then for any $\xi<\rho(\arr)$ and any $r_0\in \arr^{(\xi)}\setminus\arr^{(\xi+1)}$ there exists a full subtree $\ess$ of $\arr[r_0\preceq']$ and a family $(y_s^\ast)_{s\in\ess}\subseteq B_{Y^\ast}$ such that
\begin{equation}\label{alldaybaby}
\langle y_s^\ast,Tx_r\rangle = \langle x_s^\ast,x_r\rangle,\qquad s\in\ess,\,r\in\arr
\end{equation}
and the map $s\mapsto y_s^\ast $ from $\ess$ to $Y^\ast$ is coarse-wedge-to-weak$^\ast$ continuous.
\end{lemma}

\begin{proof}
We proceed by induction on $\xi$. For the base case, namely $\xi=0$, fix $r_0\in\arr^{(0)}\setminus\arr^{(1)}$, let $\ess=\{ r_0\}$ and choose $y_{r_0}^\ast\in B_{Y^\ast}$ such that $T^\ast y_{r_0}^\ast= x_{r_0}^\ast$. In this way we see that the assertion of the lemma is true in the case $\xi=0$.

We now address the inductive step. Suppose $0<\zeta<\rho(\arr)$ and that the assertion of the lemma holds for all $\xi<\zeta$; we will now show it is then true for $\xi=\zeta$. To this end fix $r_0\in\arr^{(\zeta)}\setminus\arr^{(\zeta+1)}$ and for each $t\in\arr[r_0+]$ let $\ess_t$ be a full subtree of $\arr[t\preceq' ]$ and $(y_s^\ast)_{s\in\ess_t}\subseteq B_{Y^\ast}$ a family such that 
\begin{equation}\label{alldayadult}
\langle y_s^\ast,Tx_r\rangle = \langle x_s^\ast,x_r\rangle,\qquad s\in\ess_t,\,r\in\arr
\end{equation}
and the map $\Xi_t: s\mapsto y_s^\ast $ from $\ess_t$ to $Y^\ast$ is coarse-wedge-to-weak$^\ast$ continuous. Let $d$ be a metric on $B_{Y^\ast}$ that is compatible with the weak$^\ast$ topology on $B_{Y^\ast}$ and let $(t_n)_{n=0}^\infty$ be an injective sequence in $\arr[r_0+]$ such that $(y_{t_n}^\ast)_{n=0}^\infty$ is weak$^\ast$ convergent. Let $y_{r_0}^\ast$ denote the weak$^\ast$-limit of $(y_{t_n}^\ast)_{n=0}^\infty$. Passing to a subsequence we may assume that $d(y_{t_n}^\ast,y_{r_0}^\ast)<1/n$ for each $n<\omega$. By \cref{coarsefacts} and the continuity of the maps $\Xi_{t_n}$, $n<\omega$, for each $n<\omega$ we may choose a finite set $\eff_n\subseteq \ess_{t_n}[t_n+]$ such that
\[
\forall\,n<\omega\quad\forall\,t\in \dubyoo_{\ess_{t_n}}(t_n,\eff_n)\quad d(\Xi_{t_n}(t),\Xi_{t_n}(t_n))<\frac{1}{n}.
\]
Define $\ess:=\{ r_0\}\cup \bigcup_{n<\omega}\dubyoo_{\ess_{t_n}}(t_n,\eff_n)$ and $\Xi:\ess\longrightarrow Y^\ast$ by
\[
\Xi(s) = \begin{cases}
y_{r_0}^\ast,& s=r_0\\
\Xi_{t_n}(s),& s\in \dubyoo_{\ess_{t_n}}(t_n,\eff_n),\, n<\omega
\end{cases}
\]
It is straightforward to check that $\ess$ is a full subtree of $\arr[r_0\preceq']$ since $\dubyoo_{\ess_{t_n}}(t_n,\eff_n)$ is a full subtree of $\ess_{t_n}$ for every $n<\omega$. To see that \cref{alldaybaby} holds for this $\ess$, note that since \cref{alldayadult} holds for $t=t_n$, for all $n<\omega$, we need only check that \cref{alldaybaby} holds in the case where $s=r_0$. To this end note that for all $r\in\arr$ we have
\[
\langle y_{r_0}^\ast,Tx_r\rangle = \lim_{n\rightarrow\omega} \langle y_{t_n}^\ast,Tx_r\rangle =  \lim_{n\rightarrow\omega}\langle x_{t_n}^\ast,x_r\rangle= \langle x_{r_0}^\ast,x_r\rangle,
\]
where the final equality follows from \cref{quiorquogger}. Thus \cref{alldaybaby} holds for all $s\in \ess$ and $r\in\arr$.

To complete the proof it remains only to establish the continuity of $\Xi$. Since each $\Xi_{t_n}$ is continuous, for $n<\omega$, the only nontrivial case to check is whether $\Xi$ is continuous at $r_0$. Fix $\lambda>0$. Let $N<\omega$ be large enough that $N\lambda >2$ and let $\eff=\{ t_0,\ldots,t_{N-1}\}$. For each $s\in \dubyoo_{\ess}(r_0,\eff)\setminus \{ r_0\}$ there exists a unique $n_s\geq N$ such that $t_{n_s}\preceq' s$. So for $s\in \dubyoo_{\ess}(r_0,\eff)\setminus \{ r_0\}$ we have
\[
d(\Xi(s),\Xi(r_0))\leq d(\Xi(s),\Xi(t_{n_s})) + d(\Xi(t_{n_s}),\Xi(r_0))<\frac{1}{n_s}+\frac{1}{n_s}\leq \frac{2}{N}<\lambda.
\]
It follows that $d(\Xi(s),\Xi(r_0))<\lambda$ for all $s\in \dubyoo_{\ess}(r_0,\eff)$, hence $\Xi$ is continuous at $r_0$ since $\dubyoo_{\ess}(r_0,\eff)$ is open in $\ess$ by \cref{coarsefacts}.
\end{proof}

\begin{proof}[Proof of \cref{sepseekay}]
Suppose $Y$ is separable. Let $(\arr,\preceq')$ be a blossomed tree with $\rho(\arr) = \rho(\tee)$ (c.f. \cref{blossexist}), let $\epsilon'>0$ be small enough that $s_{\epsilon'}^{\omega^\alpha}(T^\ast(B_{Y^\ast}))\neq\emptyset$, and let $N<\omega$ be large enough $\rho(\tee)\leq\omega^\alpha2^N+1$. Set $\epsilon=2^{-N-1}\epsilon'$, so that $s_\epsilon^{\rho(\tee)-1}(T^\ast(B_{Y^\ast}))\neq\emptyset$ by \cref{dingraniel}. \cref{birthdaybun} yields families $(x_s)_{s\in\arr}\subseteq S_X$ and $(x_s^\ast)_{s\in\arr}\subseteq T^\ast (B_{Y^\ast})$ such that 
\begin{equation}\label{biorthoggy}
\langle x_s^\ast ,x_r\rangle=\begin{cases}
\langle x_r^\ast,x_r\rangle>\frac{\epsilon}{17}&\mbox{if }r\preceq' s\\ 0 &\mbox{if }r\npreceq' s
\end{cases}, \quad r,s\in \arr\,.
\end{equation}
We apply \cref{seasonover} with $r_0$ the root of $\arr$ and $\xi=\rho(\tee)-1$ to obtain a full subtree $\ess$ of $\arr$ and a family $(y_s^\ast)_{s\in\ess}\subseteq B_{Y^\ast}$ such that
\begin{equation}\label{porridge}
\langle y_s^\ast,Tx_r\rangle = \langle x_s^\ast,x_r\rangle,\qquad s,r\in\ess,
\end{equation}
and the map $\Xi: s\mapsto y_s^\ast$ from $\ess$ to $Y^\ast$ is coarse-wedge-to-weak$^\ast$ continuous. From \cref{biorthoggy} and \cref{porridge} we deduce that
\begin{equation}
\langle y_s^\ast ,Tx_r\rangle=\begin{cases}
\langle y_r^\ast,Tx_r\rangle>\frac{\epsilon}{17}&\mbox{if }r\preceq' s\\ 0 &\mbox{if }r\npreceq' s
\end{cases}, \qquad r,s\in\ess.
\end{equation}
By an application of \cref{kaktorfar} with $K=\ess$, index set $I=\ess$, clopen sets $K_s=\ess[s\preceq']$ for $s\in\ess$, and $\delta_s=\epsilon/17$ for all $s\in\ess$, we obtain that $\sigma_\ess$ factors through $T$. So to prove the first assertion of \cref{sepseekay} it now suffices to show that $\mathring{\sigma}_\tee$ factors through $\sigma_\ess$. To this end we now define three operators, $S$, $R$, and $P$, so that $\mathring{\sigma}_\tee =PS\sigma_\ess R$. Let $\phi:\tee\longrightarrow\ess$ be an order-isomorphism of $\tee$ onto a downward-closed subtree of $\ess$, noting that such an embedding exists by \cref{blossrels}. Since $\phi$ is coarse wedge continuous by \cref{closedown}(ii), the operator $S\in\allop(C(\ess),C(\tee))$ given by setting $Sf=f\circ\phi$ for each $f\in C(\ess)$ is well-defined. Let $R\in\allop(\ell_1(\tee^\star),\ell_1(\ess))$ be the operator defined by setting $Re_t=e_{\phi(t)}^\ess$ for every $t\in\tee^\star$. Let $t_\emptyset$ denote the root of $\tee$ and define $P\in\allop(C(\tee),C_0(\tee))$ by setting $Pf=f-f(t_\emptyset)\chi_\tee$ for each $f\in C(\tee)$. Since for $t\in\tee^\star$ we have
\[
PS\sigma_\ess Re_t = PS\sigma_\ess e_{\phi(t)}^\ess= PS\chi_{\ess[\phi(t)\preceq']} = P\chi_{\tee[t\preceq]}=\chi_{\tee[t\preceq]}=\mathring{\sigma}_\tee e_t,
\]
we conclude that $\mathring{\sigma}_\tee=PS\sigma_\ess R$. The first assertion of the theorem is proved.

For the second assertion of the theorem, we now suppose that $\tee$ is blossomed and $\rho(\tee)\geq\omega^\alpha$. As $\tee$ is infinite and rooted, we have that $\rho(\tee)\geq2$ and $\rho(\tee)$ is a successor ordinal, hence $\rho(\tee)>\omega^\alpha$. Since $\tee$ is blossomed, an application of \cref{lbszlenk} yields $Sz(\Sigma_{\tee^\star})\geq\rho(\tee)>\omega^\alpha$. Moreover, as noted in the paragraph following \vref{trootedree}, $Sz(\Sigma_{\tee^\star})$ coincides with $Sz(\mathring{\sigma}_\tee)$, hence $\mathring{\sigma}_\tee$ is non-$\alpha$-Szlenk. Note also that the codomain of $\mathring{\sigma}_\tee$, namely $C_0(\tee)$, is norm separable since $\tee$ is countable. On the other hand, by the first assertion of the theorem we have that $\mathring{\sigma}_\tee$ factors through any non-$\alpha$-Szlenk operator with separable codomain, hence we conclude that $\mathring{\sigma}_\tee$ is in this case universal for the class of non-$\alpha$-Szlenk operators with separable codomain.
\end{proof}

\begin{remark}
Bourgain, in a study\autocite{Bourgain1979} of fixing properties of operators of large Szlenk index acting on $C(K)$ spaces, represented $C(L)$ spaces with $L$ countable, compact and Hausdorff as spaces of scalar-valued functions on blossomed trees. Bourgain associates to each tree $(\tee_\xi,\sqsubseteq)$ constructed in \cref{blossexist} above (subtrees of the tree $(\Omega,\sqsubseteq)$ constructed in \cref{fullcounter}) a Banach space $X_\xi$, isometrically isomorphic to $C_0(\tee_\xi)$, defined as the completion of $c_{00}(\tee_\xi^\star)$ (the space of finitely-supported scalar-valued functions on $\tee_\xi^\star$) with respect to the norm $\Vert \cdot\Vert_\xi$ defined by setting
\[
\Vert x\Vert_\xi = \sup_{t\in\tee_\xi^\star}\Big\vert \sum_{s\sqsubseteq t}x(s)\Big\vert,\qquad x\in c_{00}(\tee_\xi^\star).
\]
The main assertion of \cref{sepseekay} may be recast as follows: \emph{Suppose $\alpha,\xi<\omega_1$ are such that $\omega^\alpha\leq\xi<\omega^{\alpha+1}$ and let $T_\xi\in\allop(\ell_1(\tee_\xi^\star),X_{\tee_\xi^\star})$ denote the continuous linear extension of the formal identity map from $(c_{00}(\tee_\xi^\star),\Vert\cdot\Vert_{\ell_1(\tee_\xi^\star)})$ to $X_{\tee_\xi^\star}$. Then $T_\xi$ is universal for the class of non-$\alpha$-Szlenk operators with norm separable codomain.}
\end{remark}

We conclude the current section with some observations regarding the aforementioned universal operator theorems of Johnson and Lindenstrauss-Pe{\l}czy{\'n}ski. In particular, we note the following corollaries of \cref{LPthm,Jthm}. These results appear in \textcite{Johnson1971a} and \textcite{Lindenstrauss1968}, respectively, under the stronger hypothesis that $Y$ is norm separable. 
\begin{corollary}\label{LPcor}
Let $X$ and $Y$ be Banach spaces such that $Y$ has weak$^\ast$-sequentially compact dual ball and let $T\in\allop(X,Y)$ be non-weakly compact. Then $T$ factors the summation operator $(a_n)_{n=1}^\infty \mapsto (\sum_{i=1}^n a_i)_{n=1}^\infty$ from $\ell_1$ to $c$.
\end{corollary}
\begin{proof} 
By \cref{LPthm} there exist $U\in\allop(\ell_1,X)$ and $V\in\allop(Y,\ell_\infty)$ such that $VTU$ is the summation operator from $\ell_1$ to $\ell_\infty$. For $n\in\nat$ let $f_n^\ast$ denote the $n$th coordinate functional on $\ell_\infty$; that is, $f_n^\ast (f) =f(n)$ for every $f\in\ell_\infty$.
Since $Y$ has weak$^\ast$-sequentially compact dual ball there is a weak$^\ast$-convergent subsequence $(V^\ast f_{n_k}^\ast)_{k=1}^\infty$ of $(V^\ast f_{n}^\ast)_{n=1}^\infty$. Define $A\in\allop( \ell_1, X)$ by setting $Ae_k = Ue_{n_k}$ for each $k\in\nat$ a define $B\in\allop(Y,c)$ by setting $By=(\langle V^\ast f_{n_k}^\ast, y\rangle)_{k=1}^\infty$ for each $y\in Y$. For $k,l\in\nat$ we have \[
(BTAe_k)(l) = \langle V^\ast f_{n_l}^\ast, TAe_k\rangle =\langle f_{n_l}^\ast, VTUe_{n_k}\rangle = \begin{cases}
1,&l\geq k\\
0,&l<k
\end{cases},
\]
hence $BTA$ coincides with the summation operator from $\ell_1$ to $c$.
\end{proof}

\begin{corollary}\label{Jcor}
Let $X$ and $Y$ be Banach spaces such that $Y$ has weak$^\ast$-sequentially compact dual ball and let $T\in\allop(X,Y)$ be non-compact. Then $T$ factors the formal identity mapping from $\ell_1$ to $c_0$.
\end{corollary}

\begin{proof}
By \cref{Jthm} there exist $U\in\allop(\ell_1,X)$ and $V\in\allop(Y,\ell_\infty)$ such that $VTU$ is the formal identity mapping from $\ell_1$ to $\ell_\infty$. For $n\in\nat$ let $f_n^\ast$ denote the $n$th coordinate functional on $\ell_\infty$.
Since $Y$ has weak$^\ast$-sequentially compact dual ball there is a  subsequence $(V^\ast f_{n_k}^\ast)_{k=1}^\infty$ of $(V^\ast f_{n}^\ast)_{n=1}^\infty$ converging weak$^\ast$ to some $y^\ast\in Y^\ast$. Define $A\in\allop( \ell_1, X)$ by setting $Ae_k = Ue_{n_k}$ for each $k\in\nat$ and define $B\in\allop(Y,c_0)$ by setting $By=(\langle V^\ast f_{n_k}^\ast - y^\ast, y\rangle)_{k=1}^\infty$ for each $y\in Y$. Since \[ \forall\,k\in\nat\quad \langle y^\ast , TUe_{n_k}\rangle = \lim_{l\rightarrow\infty} \langle V^\ast f_{n_l}^\ast , TU{e_{n_k}}\rangle= \lim_{l\rightarrow\infty}\langle f_{n_l}^\ast , VTU{e_{n_k}}\rangle =0,\] it follows that for $k,l\in\nat$ we have\[
(BTAe_k)(l)= (BTUe_{n_k})(l) = \langle V^\ast f_{n_l}^\ast-y^\ast, TUe_{n_k}\rangle =\langle f_{n_l}^\ast, VTUe_{n_k}\rangle = \begin{cases}
1,&l= k\\
0,&l\neq k
\end{cases},
\]
hence $BTA$ coincides with the formal identity mapping from $\ell_1$ to $c_0$.
\end{proof}

We conclude the current section of the paper with the following open question.
\begin{question} Does the statement of \cref{sepseekay} remain true if the condition that $Y$ be norm separable is relaxed and $Y$ is only assumed to have weak$^\ast$-sequentially compact dual ball?
\end{question}

\section{Uncountable Szlenk indices and universality}\label{nonspear}

The proof of \cref{countuniv} makes use of the fact that an operator $T$ can fail to be $\alpha$-Szlenk for a given ordinal $\alpha$ in essentially two different ways. More precisely, it can be that $Sz(T)$ is defined and larger than $\omega^\alpha$, or it can be that $Sz(T)$ is undefined. In particular, a non $\alpha$-Szlenk operator can be either Asplund or non-Asplund. This observation and the first assertion of \cref{countuniv} lead naturally to the following open question.
\begin{question}\label{openquest}
Let $\alpha\geq \omega_1$ be an uncountable ordinal. Does $\asplundop\cap\complement\szlenkop_\alpha$ admit a universal element?
\end{question}

In light of the approach taken in \cref{countuniv} to prove the existence of universal elements of $\complement\szlenkop_\alpha$ for $\alpha<\omega_1$, it is natural to guess that a first step towards answering \cref{openquest} could involve consideration of operators of the form $\Sigma_\tee$, where $(\tee,\preceq)$ is a tree, and extending the definition of a blossomed tree (c.f. \cref{gaspardefn}) to the uncountable setting as follows: say that a tree $(\tee,\preceq)$ is \textit{blossomed} if it is rooted, well-founded, and for every $t\in\tee\setminus\MAX(\tee)$ there exists a bijection $\psi_t: \max\{ \omega, cof(\rho_\tee(t))\}\longrightarrow \tee[t+]$ such that $\zeta\leq\zeta'<\max\{ \omega, cof(\rho_\tee(t))\}$ implies $\rho_\tee(\psi_t(\zeta))\leq \rho_\tee(\psi_t(\zeta'))$. Examples of such trees $\tee$ with $\rho(\tee)=\xi+1$ for a given ordinal $\xi$ may be obtained via a similar construction to that provided in \cref{blossexist}, but with $\tee$ consisting of finite sequences of ordinals (ordered by extension, as in \cref{blossexist}). Moreover, under this more general definition of a blossomed tree, a blossomed tree $(\tee,\preceq)$ satisfies $\rho(\tee)<\omega_1$ if and only if $\tee$ is countable and satisfies the usual definition of blossomed tree given in \cref{gaspardefn}. The natural candidate for a universal element of $\asplundop\cap\complement\szlenkop_{\omega_1}$ under this approach is $\Sigma_\ess$, where for each $\alpha<\omega_1$ we let $(\tee_\alpha,\preceq_\alpha)$ be a blossomed tree with $\rho(\tee_\alpha)=\alpha+1$ and set $\ess= \bigcup_{\alpha<\omega_1}(\{ \alpha\} \times \tee_\alpha)$, with an order $\preceq$ on $\ess$ defined by setting $(\alpha,t)\preceq(\alpha',t')$ if and only if $\alpha=\alpha'$ and $t\preceq_\alpha t'$. We do not know the answer to \cref{openquest}, even in the case $\alpha=\omega_1$. However, as we shall now see, operators of the form $\Sigma_\tee$, where $(\tee,\preceq)$ is a well-founded tree - cannot be expected to provide absolute examples of universal elements of the classes $\asplundop\cap\complement\szlenkop_\alpha$ for $\alpha\geq\omega_1$ in general. In particular, we shall see that it is consistent with ZFC that the operator $\Sigma_\ess$ defined above is not universal for $\asplundop\cap\complement\szlenkop_{\omega_1}$.

Let $Z$ be a Banach space and $(\tee,\preceq)$ a tree. Let \[ \mathcal{O}_\tee=\{ t\in\tee\mid ht_\tee(t)=0 \mbox{ or }ht_\tee(t)\mbox{ is a successor}\},\] noting that $\mathcal{O}_\tee=\tee$ if and only if $ht(\tee)\leq\omega$. In particular, $\mathcal{O}_\tee=\tee$ whenever $\tee$ is well-founded. If $\Sigma_\tee$ factors through $Z$ then, by \cref{factorkar}, there exist $\delta>0$, $(x_t)_{t\in\tee}\subseteq Z$ and $(x_t^\ast)_{t\in\tee}\subseteq Z^\ast$ satisfying \cref{deltadelta}. It follows that $Z$ admits a biorthogonal system of cardinality $\vert\mathcal{O}_\tee\vert$, namely $(x_t, z_t^\ast)_{t\in\mathcal{O}_\tee}$, where 
\[ z_t^\ast = \begin{cases}x_t^\ast-x_{t^-}^\ast & \mbox{if }ht_\tee(t)>0\\ x_t^\ast& \mbox{if }ht_\tee(t)=0\end{cases}, \quad t\in\mathcal{O}_\tee.\]
(For background on biorthogonal systems in Banach spaces we refer the reader to the book \textcite{Hajek2008}.) The following proposition is now immediate.
\begin{proposition}\label{nononon}
Let $Z$ be a Banach space not admitting an uncountable biorthogonal system and let $(\tee,\preceq)$ be a tree such that $\mathcal{O}_\tee$ is uncountable. Then $\Sigma_\tee$ does not factor through $Z$.
\end{proposition}
It is consistent with ZFC that there exists an Asplund space $W$ with $Sz(W)>\omega_1$ and $W$ does not admit an uncountable biorthogonal system. Thus, by \cref{nononon}, it is consistent with ZFC that the operator $\Sigma_\ess$ defined earlier in the current section is not universal for $\asplundop\cap\complement\szlenkop_{\omega_1}$. An example of a compact Hausdorff space $K$ such that $C(K)$ is such a space $W$ was constructed in the 1970s by Kunen, though the construction was not published until much later by Negrepontis\autocite{Negrepontis1984}. (For further historical remarks concerning the existence of uncountable biorthogonal systems, see Remark 4 of \textcite{Todorcevic2006}.) Since a Banach space $C(L)$ is Asplund if and only if $L$ is scattered\autocite{Namioka1975}, the space $C(K)$ arising from Kunen's construction is Asplund. Moreover, the Cantor-Bendixson rank of Kunen's space $K$ is larger than $\omega_1$. Thus, the $C(K)$ space arising from Kunen's construction is indeed an example of such a space $W$ once we have observed the following fact: for $L$ a compact Hausdorff space the Szlenk index of $C(L)$ is bounded below by the Cantor-Bendixson rank of $L$. This is an easy consequence of the well-known fact that the mapping that takes $l\in L$ to the evaluation functional of $C(L)$ at $l$ is a homeomorphic embedding with respect to the weak$^\ast$ topology, and the image of $L$ under this embedding is a $1$-separated subset of $B_{C(L)^\ast}$. From this fact it is easy to see that $Sz(C(L),1)$ is bounded below by the Cantor-Bendixson rank of $L$, hence $Sz(C(L))$ is bounded below by the Cantor-Bendixson rank of $L$.

Finally, we mention a more recent construction of Brech and Koszmider\autocite{Brech2011}, who establish the consistency of a scattered compact Hausdorff space $J$ having Cantor-Bendixson rank equal to $\omega_2+1$ and such that $C(J)$ does not admit an uncountable biorthogonal system. If $\tee$ is a tree that is blossomed in the generalised sense introduced at the beginning of the current section, and if $\rho(\tee)=\omega_2+1$, then $Sz(\Sigma_\tee)=\omega_2\omega$. However, by \cref{nononon}, $\Sigma_\tee$ does not factor through the Brech-Koszmider space $C(J)$ which satisfies $Sz(C(J))\geq\omega_2\omega$.

\section*{Acknowledgements}
The author is grateful to Ryan Causey for pointing out the relevance of the references \textcite{Alspach2005} and \textcite{Beanland2018} to the current paper. The author also thanks the anonymous referee(s) for numerous helpful comments that have significantly improved the quality of the paper, including in particular for suggesting a significant simplification to the proof of \cref{zippinbase}.

\printbibliography

\end{document}